\documentclass[11pt,a4paper]{article}   		

\usepackage{fullpage}

\usepackage[bbgreekl]{mathbbol}
\usepackage[bb=boondox]{mathalfa}
\usepackage{amssymb}
\usepackage{graphicx,mathrsfs,amsmath,amsfonts,appendix,color,amsthm}
\usepackage{bm,bbm}

\usepackage{enumitem}
\usepackage{pifont}
 \usepackage{tikz}
\usepackage[english]{babel}
\usepackage{caption}
\usepackage{subcaption}
\usepackage{array}
\newcolumntype{M}[1]{>{\centering\arraybackslash}m{#1}}
\usepackage{algorithm,algorithmic}

\newtheorem{remark}{Remark}
\newtheorem{theorem}{Theorem}
\newtheorem{lemma}{Lemma}
\newtheorem{corollary}{Corollary}

\newcommand{\R}{\mathbb{R}}

\newcommand{\Gin}{\Gamma_{in}}
\newcommand{\Gout}{\Gamma_{out}}
\DeclareMathOperator{\Div}{div}
\newcommand{\be}{\begin{equation}}
\newcommand{\en}{\end{equation}}
 \newcommand{\ben}{\begin{equation*}}
 \newcommand{\enn}{\end{equation*}}

\title{Parameter estimation for macroscopic pedestrian dynamics models from
microscopic data}
\author{Susana N. Gomes \thanks{University of Warwick, Coventry CV4 7AL, UK (susana.gomes@warwick.ac.uk)}
\and Andrew M. Stuart \thanks{California Institute of Technology, 1200 E. California Blvd, Pasadena, CA 9112. } 
\and  Marie-Therese Wolfram \thanks{University of Warwick, Coventry CV4 7AL, UK and RICAM, Austrian Academy of Sciences, Altenbergerstr. 66, 4040 Linz, AT }}

\begin{document}
	\maketitle
	\begin{abstract}
		In this paper we develop a framework for parameter estimation in macroscopic pedestrian models using individual trajectories -- microscopic data. We consider a unidirectional flow 
		of pedestrians in a corridor and assume that the velocity decreases with the average density according to the fundamental diagram. Our model is formed from a coupling between a density dependent stochastic differential equation and a nonlinear partial differential equation
		for the density, and is hence of McKean--Vlasov type. 
		We discuss identifiability of the parameters appearing in the fundamental
		diagram from  trajectories of individuals, and we introduce  optimization 
		and Bayesian methods to perform the identification. We  analyze the performance of the developed methodologies in 
		various situations, such as for different in- and outflow conditions, for
		varying numbers of individual trajectories and for differing channel geometries.
		{\bf AMS subject classifications:} 
	\end{abstract}
	
	\section{Introduction}
	\label{sec:intro}
	The complex dynamics of large pedestrian crowds as well as constant urbanization has initiated many developments in 
	the field of transportation research, physics, urban planning and more recently applied mathematics. Originally, model 
	development was mostly based on empirical observations. More recently
	a plethora of data, from video surveillance 
	cameras as well as experiments, is starting to become available. This data, usually individual trajectories, is used to calibrate individual based models or
	to identify characteristic relations, such as the fundamental diagram, 
	which have the potential to be observed in large pedestrian crowds. 
	In this paper we study such a calibration approach, setting up parameter estimation methodologies which allow us to estimate parameters
	in macroscopic pedestrian models using individual trajectories. 
	
	There is a rich literature on mathematical models for pedestrian dynamics, see \cite{CPT2014}. Individual based, also known as microscopic 
	models, are among the most popular in the engineering and transportation literature. In these models individuals are 
	characterized by their position and sometimes velocity, see for example \cite{Helbing1995}. On the macroscopic level the crowd is described by a density. These models are often derived from first principles, since the rigorous transition from the
	micro- to the macroscopic model is still an open problem in certain scaling regimes.  {Then more general interactions, as proposed in \cite{aggarwal2015nonlocal,colombo2012nonlocal,goatin2015mixed,hughes2003flow}, can be considered. In this paper we focus on mean-field interactions only. Here individual dynamics are influenced by the averaged behavior of the crowd which can be characterized by the average velocity, flow or density.}
	{Often, for example in high density regimes, individual dynamics have little influence on the overall flow. For this reason averaged quantities are of specific interest to characterize the crowd behavior}.  While individual based models depend on a large number of parameters, which allow a detailed description of individual trajectories, their identification and 
	estimation is extremely challenging and computationally costly. Hence we focus on the identification of macroscopic relations from
	individual trajectories in the following.
	
	An important and commonly used macroscopic relation is the so-called \textit{fundamental 
		diagram}. The fundamental diagram relates an averaged observed pedestrian density to either the measured velocity or 
	outflow.  It is a well-established characteristic quantity in vehicle traffic flow theory, see \cite{lighthill1955}, but also plays an important role in pedestrian dynamics. {Examples of its use include the quantification of the capacity of pedestrian facilities or the evaluation of pedestrian models, see \cite{seyfried2005}}.
	To understand the latter we may use trajectories which are collected in controlled experiments. These experiments are conducted for various conditions -- 
	different domains (corridors, junctions, ...), uni- and bidirectional flows, varying inflow and outflow 
	conditions, in addition to considering the effects of
	diverse social and cultural backgrounds; see for example \cite{boltes2010}. In  experiments trajectory 
	recordings usually start once the experiment has equilibrated. The collected video or sensor data is used to extract individual trajectory data. 
	More recently also images from motion sensing cameras, such as Kinect have been used, see \cite{Corbetta2017}, to 
	collect individual trajectories in public spaces. The fundamental diagram is then calculated by evaluating the individual 
	trajectory data using suitable averaging techniques. In unidirectional flows different regimes can be observed -- at very low densities pedestrians walk with 
	a (maximum) velocity (denoted by $v_{\max}$ later on).
	Then the average speed decreases as the density increases. At a certain point the velocity approaches zero due to overcrowding. 
	We refer to the density at which overcrowding occurs as $\rho_{\max}$ later on.
	Different values for $\rho_{\max}$ and $v_{\max}$ can be found in the literature, ranging  from $3.8$ to $10$ pedestrians per m$^{2}$ for $\rho_{\max}$ and $0.98$ m/s to $1.5$ m/s for $v_{\max}$.
	These deviations can be explained by the experimental setup and measurement techniques used, as well as psychological 
	and cultural factors, cf. \cite{chattaraj2009,seyfried2005, zhang2012}.
	
	In this paper we consider a unidirectional flow in a corridor as illustrated in Figure \ref{f:sketch}. We assume that 
	pedestrians enter the corridor on the left through $\Gin$ and exit on the right through $\Gout$, with no possible 
	entrance or exit at the walls $\Gamma_N$.
	\begin{figure}
		\begin{subfigure}[b]{0.5\textwidth}
			\floatstyle{plain}
			\begin{tikzpicture}[scale=1.4]
			\node[inner sep = 0pt] (pedestrian1) at (1.5,1)
			{\includegraphics[width=0.05\textwidth]{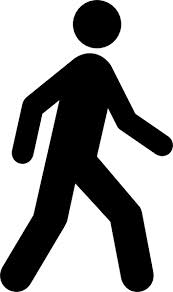}};
			\draw[thick] (1,0) -- (4,0) node [pos=0.5,below] {{\footnotesize $\Gamma_N$}} ;
			\draw[thick] (4,0) -- (4,2) node [pos=0.5,right  ] {{\footnotesize $\Gout$}} ;
			\draw[thick] (4,2) -- (1,2) node [pos=0.5,above] {{\footnotesize $\Gamma_N$}} ;
			\draw[thick] (1,2) -- (1,0) node [pos=0.5,left] {{\footnotesize $\Gin$}} ;
			\node  at (0.5,-0.10) {\small{$(0,-\ell)$}};
			\node  at (4.5,-0.1) {\small{$(L,-\ell)$}};
			\node  at (0.5,2.20) {\small{$(0,\ell)$}};
			\node at (4.5,2.20) {\small{$(L,\ell)$}};
			\draw[thick,->] (1.35,0.2) -- (1.75,0.2) node [pos=0.95,below] {{\footnotesize $x_1$}};
			\draw[thick,->] (1.35,0.2) -- (1.35,0.60)  node [pos=0.95,left] {{\footnotesize $x_2$}};
			\draw[thick,->] (1.75,1.0) -- (2.25,1.0);
			\end{tikzpicture}
			\caption{Sketch of the domain}\label{f:sketch}
		\end{subfigure}
		\floatstyle{plain}
		\begin{subfigure}[b]{0.4\textwidth}
			\includegraphics[width=\textwidth]{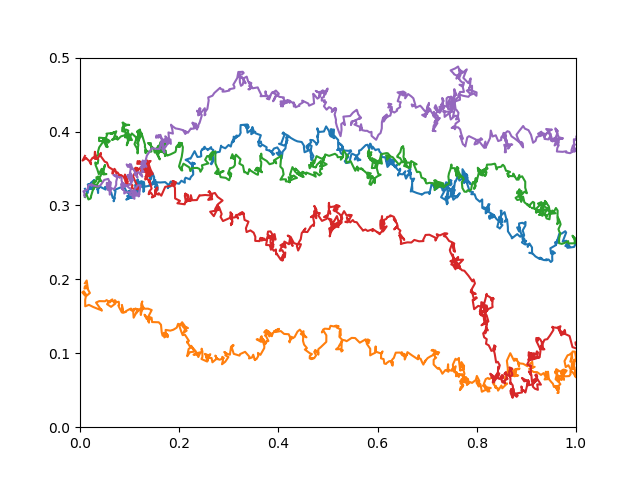}
			\caption{Sample of 5 generated trajectories.\label{f:traj-intro}}
		\end{subfigure}
		\caption{Modeling setup considered}
	\end{figure}
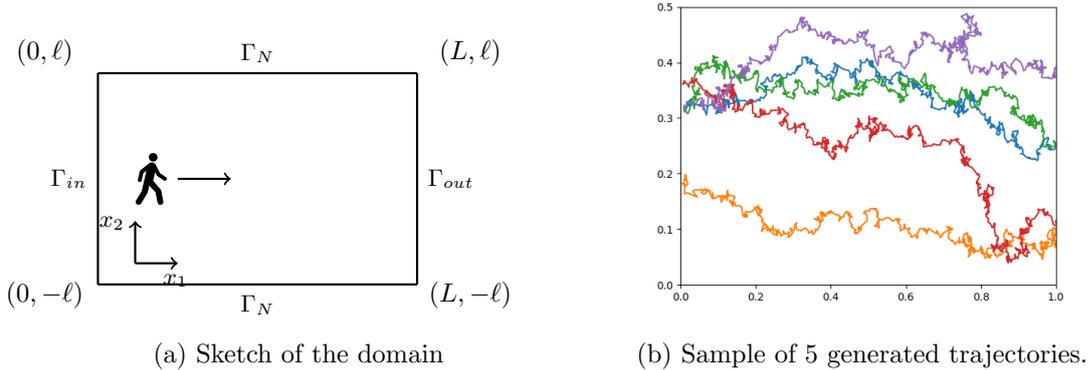
	The evolution of the overall density $\rho = \rho(x,t)$ can be modeled by a nonlinear Fokker--Planck equation, 
	with (nonlinear) mixed boundary conditions, describing the inflow and outflow at the entrance and exit. Here the drift term 
	accounts for the nonlinear velocity dependence on the density, as described 
	by the fundamental diagram. The scaled
	density $\rho$ allows us to describe individual trajectories $X = X(t)$ as realizations of a generalized McKean--Vlasov process.
	This introduces a coupling between the (microscopic) 
	trajectories $X$ and the (macroscopic) density $\rho$ of the system, with $X$ 
	governed by a density dependent stochastic differential equation (SDE) \eqref{eq:sde}
	and $\rho$ governed by a Fokker-Planck 
	type partial differential equation (PDE) \eqref{eq:FP}. 
	It is important to note that the density $\rho$ is not a 
	probability density function, since the total mass changes in time due to the inflow and outflow boundary conditions. 
	Rather than directly employ the empirical density of the process to perform 
	parameter estimation we instead try to learn parameters from individual
	trajectories whose guiding velocity depends on the density via the
	fundamental diagram; this avoids the difficult problem of creating
	smooth densities from collections of individual trajectories.  
	We study the coupling between the PDE and the trajectories of each pedestrian 
	with the goal of learning two parameters from individual trajectories of the resulting nonlinear Markov process: the maximum speed $v_{\max}$ and the maximum density $\rho_{\max}$. Whilst there is some work concerning
	parameter estimation of a nonlinear Markov process this is primarily
	focused on settings in which the data concerns macroscopic or mean
	properties \cite{giesecke2017inference}; see also the recent analyses
	of nonlinear Markov processes in \cite{garnier2016mean,garnier2017consensus}.
	In contrast we focus on learning
	from microscopic data; closely related work concerns the learning
	of (macroscopic) ocean properties from Lagrangian (microscopic)
	trajectories \cite{apte2008bayesian}.

	We will see that the maximum density $\rho_{\max}$ is not identifiable in our setup. The parameter $v_{\max}$ is learnable and will be
	determined using a maximum a posteriori (MAP) estimator as well as a fully Bayesian methodology. 
	Estimates such as MAP (as well as the related maximum likelihood estimator, MLE) are asymptotically unbiased and 
	asymptotically normal as long as the path observations are sufficiently long (i.e., observe the evolution of the generalized 
	McKean--Vlasov process for $T\rightarrow\infty$)~\cite{Kutoyants2004,prakasarao}. However, since  individuals leave the 
	domain after a finite time, we use multiple trajectories and show that this leads to similar properties to those found when observing one very long 
	trajectory of an SDE.

	The considered mathematical model is rather simplistic from a modeling
	perspective. At the moment we do not account for many known features 
	of individual behavior, such as collision avoidance or small group dynamics.  However, the proposed setup is an 
	appropriate model for the density of the process and still mathematically tractable. The developed framework will allow 
	us to  extend and generalize our parameter learning approach to more complicated models, geometries, and situations in the future.
	Furthermore we observe that the SDE model considered describes individual motion only. The (Brownian) noise models the
	tendency of individuals to not move along completely straight lines. However, it does not correspond to measurement errors in experiments.
	Including such errors would involve an extra equation for the trajectory observations and requires the use of 
	filtering methods. This problem could be considered in the future; however
	it is likely that the model error itself is larger than the observation error
	in such situations, and investment in refining the model is arguably
	more important than accounting for noise in the trajectory data.
	
	This paper builds a systematic methodology to identify parameters of interest 
	in systems which are modeled by SDEs which depend on their own density. 
	The methodology is robust to improvements in the model, and can be extended
	to more complex sensing of individual trajectories, and to the
	incorporation of noise into the trajectory data. {This approach can also be used to identify functions instead of parameters and quantify uncertainty in our estimates. There has been an increasing interest in developing connections between trajectory data and mathematical models; see for example \cite{bode2019emergence}.
		However, in our approach the likelihood is constructed explicitly, using properties
		of stochastic differential equations, allowing for fully Bayesian inference,
		and avoiding the need for approximate Bayesian inference,  
		as employed in the recent work by Bode et al. \cite{bode2019emergence}.}
	
	Our main contributions are the following:
	\begin{itemize}
		\item set up of parameter estimation methodologies for SDEs which depend on their 
		own density, with trajectories confined to  bounded domains via boundary conditions;
		\item discussion of identifiability in the considered model of pedestrian dynamics;
		\item development of numerical methods which allow for parameter estimation
		via optimization and for uncertainty quantification via a Bayesian approach;
		\item analyzing the impact of flow conditions (inflow vs. outflow rate, transient vs. steady state profiles) on the parameter estimation; 
		\item comparison between transient and steady state density regimes,
		providing novel insights about the applicability of the fundamental diagram in transient regimes (we recall that the fundamental diagram is obtained from equilibrated experimental
		data).
	\end{itemize}
	
	The paper is organized as follows. We start by introducing the underlying mathematical models in Section~\ref{sec:models}. 
	Next we present analytic results for the PDE and SDE models in Section \ref{sec:analysis} and introduce the parameter 
	identification problem in Section \ref{sec:estimation}. The computational complexity of the considered identification 
	problem requires efficient and robust numerical solvers for the PDE and SDE models. We discuss these solvers as well as the optimization methodologies used in Section \ref{sec:comp}.
	We conclude by discussing the identifiability of $v_{\max}$ in various density regimes and for different experimental setups with numerous computational experiments in Section \ref{sec:numerical}. 
	We present our conclusions in Section~\ref{sec:conclusion}. Further details on the analysis of the PDE will be given in the Appendix.
	
	%
	%
	
	\section{The microscopic and the macroscopic models}
	\label{sec:models}
	
	We consider the uni-directional flow of a large number of pedestrians moving in a corridor $\Omega=\{(x_1,x_2) \in [0,L] \times [-\ell,\ell]\}$ as illustrated in Figure \ref{f:sketch}. 
	On the macroscopic level the evolution of the density $\rho(x,t)$ can be described by a nonlinear Fokker--Planck (FP) equation 
	\begin{equation}
		\label{eq:FP}
		\partial_t \rho(x,t)= \Div(\Sigma \nabla \rho(x,t) -\rho(x,t) F(\rho) ),
	\end{equation}
	where the diffusion matrix is of the form $ \Sigma = \operatorname{diag}(\sigma^2_1,\sigma_2^2)$. The diffusion accounts for the tendency of pedestrians to not move along 
	perfectly straight lines, and its structure allows for different diffusivities in the vertical and horizontal direction.
	Since individuals move from the left to the right, we choose a convective field $F: \R \mapsto \R^2$ of the form
	\begin{align}\label{e:F1positive}
		F(\rho)=  f(\rho) e_1, \text{ where } f: \R^+ \mapsto \R^+ \text{ and } e_1 = \begin{pmatrix}1\\  0\end{pmatrix}.
	\end{align}
	This convective field depends on density only, and will be evaluated pointwise
	to drive individual trajectories.\footnote{Later we will replace $e_1$, the
		direction of motion,  by
		a gradient vector field to allow for more complex domains such as bottlenecks;
		this formulation will be used in the existence and uniqueness proof of the PDE as well as the bottleneck example presented in Section \ref{sec:numerical}.} 
	We choose {the velocity as }
	\begin{align}\label{eq:F}
		f(\rho) = v_{\max} \left(1-\frac{\rho}{\rho_{max}}\right),
	\end{align}
	which corresponds to the density-flow relation observed in the fundamental diagram in traffic flow and pedestrian dynamics.
	Thus the model states that all individuals want to move at a maximum speed $v_{\max}$ in the absence of other pedestrians, and the actual velocity decreases 
	linearly from this value with the density $\rho$ evaluated locally. 
	We assume that the corridor is initially empty, that is $\rho(x,0) = 0$ for all $x \in \Omega$. 
	Individuals enter at a certain rate $a$ on the left through the entrance $\Gamma_{in}$ and leave on the
	right through the exit $\Gamma_{out}$ with a rate $b$.  On the rest of the boundary we impose no flux boundary conditions. 
	Let $j = - \Sigma \nabla \rho + F(\rho) \rho$ denote the flux. Then the corresponding boundary conditions are given by
	\begin{subequations}\label{e:bc}
		\begin{align}
			\label{eq:BC-inflow} j \cdot n &= -a \bigl(\rho_{max}-\rho\bigr),&\textrm{ for all \  } (x_1,x_2)  \in \Gin, \\
			\label{eq:BC-outflow} j \cdot n &=  b \rho,  &\textrm{ for all } (x_1,x_2)  \in \Gout,  \\
			\label{eq:BC-vertical} j \cdot n &= 0,   &\textrm{ for all \ } (x_1,x_2)  \in \Gamma_N,
		\end{align}
	\end{subequations}
	where 
	\begin{align*}
		\Gin &= \left\{(x_1,x_2) \in \Omega : \: x_1 = 0\right\}, ~\Gout = \left\{(x_1,x_2) \in \Omega : \: x_1 = L\right\}\\
		\Gamma_N &= \left\{(x_1,x_2) \in \Omega : \: x_2 = \pm \ell\right\}
	\end{align*} 
	and $n$ denotes the unit outer normal vector.
	Note that the inflow condition \eqref{eq:BC-inflow} includes the additional factor $\left(\rho_{max}-\rho\right)$ due 
	to volume exclusion. This prefactor arises in the formal limit, when particles are only allowed to enter or move to a 
	certain position if enough physical space is available, see \cite{Wood2009}.
	We note that $a$ and $b$ are rates of entrance and exit in the domain. Balancing the left- and righthand sides of equations~\eqref{eq:BC-inflow} and~\eqref{eq:BC-outflow} 
	gives us their units -- {both $a$ and $b$} are in $m/s$. 
	A small modification of the maximum principle calculations in \cite{Burger2016} shows that it is sufficient for the problem to be well posed that $0\leq a,b\leq v_{\max}$.	
	
	{In the following we will see that trajectory data generated
		by the model \eqref{eq:FP} is independent of the parameter $\rho_{max}$.}  Let $\tilde{\rho} = \frac{\rho}{\rho_{max}}$, then equation \eqref{eq:FP} can be rescaled as
	\begin{subequations}\label{e:system-rescaled}
		\begin{align}
			\label{e:FP-rescaled}\partial_t \tilde\rho = \Div\left(\Sigma\nabla\tilde\rho - {\tilde\rho}{\tilde F}({\tilde\rho})\right)
		\end{align}
		where 
		\begin{align}\label{e:F1positive2}
			{\tilde F}(\tilde\rho)=  {\tilde f}(\tilde\rho) e_1, \quad
			{\tilde f}(\tilde\rho) = v_{\max} \left(1-{\tilde \rho}\right).
		\end{align}
		The system is supplemented with the boundary conditions
		\begin{align}
			\label{eq:BC-inflow-rescaled} \tilde{j} \cdot n &= -a\bigl(1-\tilde\rho\bigr),&\textrm{ for all \  } (x_1,x_2)  \in \Gin, \\
			\label{eq:BC-outflow-rescaled} \tilde{j} \cdot n &= b \tilde\rho,  &\textrm{ for all } (x_1,x_2)  \in \Gout,  \\
			\label{eq:BC-vertical-rescaled} \tilde{j} \cdot n &= 0,   &\textrm{ for all \ } (x_1,x_2)  \in \Gamma_N.
		\end{align}
	\end{subequations}
	Here the scaled flux is given by 
	$\tilde{j} = -\Sigma\nabla\tilde\rho + v_{max}(1-\tilde\rho)\tilde\rho e_1$. 
	We see that the maximum density $\rho_{\max}$ is not present in the scaled formulation. Furthermore the convective field $F$, which governs individual
	trajectories, is given by \eqref{e:F1positive}, \eqref{eq:F} and can be
	expressed entirely in terms of $\tilde \rho$, with no reference to $\rho_{max}.$
	
	{Our stated aim is to identify $v_{\max}$ and $\rho_{max}$ in $f(\rho)$ \eqref{eq:F} using individual trajectories. However the preceding argument shows
		that the parameter $\rho_{\max}$ does not influence the SDE for trajectories. Since
		our parameter inference is based only on these trajectories this means that $\rho_{max}$
		cannot be learned from the data available to us. 
		Hence we make the following remark:}
	
	\begin{remark}
		\label{rem:ref}
		The parameter $\rho_{max}$ can not be identified within our adopted 
		microscopic macroscopic data-model framework. {This limitation is not caused by the identification methodologies proposed, but rather by the invariance of the 
			PDE-SDE model to scaling in $\rho$ and the fact that the measured trajectory data
			is independent of the value of $\rho_{\max}.$ Indeed we initially conducted numerical 
			experiments using the unscaled model, which led to the understanding
			that $\rho_{\max}$ is not identifiabile. Similar identifiability issues are well-known in the literature relating
			to inference for diffusion processes; see for example the paper 
			\cite{roberts2001inference} in which
			it is shown that the quadratic variation of sample paths cannot contain information
			about time-rescaling.} 
	\end{remark}
	
	{For this reason the rest of this paper is based only on the scaled Fokker--Planck 
		equation for $\tilde\rho$. However we wish to drop the $\sim$ notation for ease 
		of presentation. This corresponds to using equations 
		\eqref{eq:FP}--\eqref{e:bc} with $\rho_{\max}=1.$ Since the preceding arguments
		show that $\rho_{\max}$ is not identifiable from trajectory data, making {\em any}
		specific choice of $\rho_{\max}$, including the value $1$, will not affect
		the inference for $v_{\max}.$ The parameter $\rho_{\max}$ thus
		plays no further part in the paper.}

	Our focus, then, is on identification of $v_{\max}$. We will use the scaled Fokker--Planck equation \eqref{e:system-rescaled} in the following, which coincide with the system \eqref{eq:FP}--\eqref{e:bc} when setting $\rho_{max}=1$.

	The existence of steady states as well as the different stationary regimes -- so-called influx limited, outflux limited and maximum current regime -- are
	discussed by 
	Burger and Pietschmann in \cite{Burger2016}. We present corresponding
	existence results for the time dependent problem in Section \ref{sec:analysis}.
	Note that the steady state as well as the time dependent solutions {{ of~\eqref{e:system-rescaled}} satisfy  $0\leq\rho\leq 1$, which ensures that
	$f$ stays non-negative for all $x \in \Omega$ and times $t > 0$.
	This also allows us to define individual trajectories as realizations of the following generalized McKean--Vlasov equation
	\begin{equation}
		\label{eq:sde}
		dX(t)=F\bigl(\rho(X(t),t)\bigr) dt +\sqrt{2\Sigma} dW(t),
	\end{equation}
	where $W(\cdot)$ is a unit Brownian motion and $\rho$ {{solves
		equation \eqref{e:system-rescaled}. }
	These realizations correspond to individual trajectories as we see in Figure~\ref{f:traj-intro}. 
	
	In the presented PDE model the boundary conditions drive the dynamics of the process. Hence they need to be included at the SDE level as well if we
	wish for self-consistency. Boundary conditions for SDEs are a delicate issue and a general existence theory is not available in closed domains. It is out of the scope of this paper to state the SDE problem in a rigorous manner (e.g., using local time~\cite{Skorokhod}). However, it is important to implement the boundary conditions in a consistent way with the PDE, and we discuss
	their implementation in more detail in Section~\ref{s:SDE}.
	Once this is done we note that the scaled FP equation is self-consistent with the forward Kolmogorov equation for this SDE.
	Indeed any sufficiently regular solution $\rho$ of the Fokker--Planck equation 
	\eqref{e:system-rescaled} can be used as a drift in \eqref{eq:sde}. The corresponding Fokker--Planck equation
	is then a linear PDE with the same solution as the nonlinear PDE \eqref{e:system-rescaled}.
	See \cite[Thm. 1]{oelschlager1984martingale},\cite[Thm. 7.3.1]{ermakov1989random},\cite{dawson1983critical,gartner1988mckean} for theory relating to such PDEs
	in related settings.

	\section{Analysis of the model}
	\label{sec:analysis}
	In this section we discuss the analysis of the generalized McKean--Vlasov SDE~\eqref{eq:sde} as well as the {{ rescaled} Fokker--Planck equation \eqref{e:system-rescaled}. 
	We start with  the existence and regularity results for the Fokker--Planck 
	equation, as they ensure the well-posedness of the corresponding generalized McKean--Vlasov process.
	{{ We re-emphasize the important fact, discussed prior to
		Remark \ref{rem:ref}, that \eqref{e:system-rescaled} is identical to 
		equation \eqref{eq:FP}--\eqref{e:bc}, with $\rho_{max}=1$.} 
	
	\subsection{Analysis of the Fokker--Planck equation}
	
	The nonlinear boundary conditions \eqref{eq:BC-inflow-rescaled} and \eqref{eq:BC-outflow-rescaled} strongly influence the time dependent and steady state solutions. 
	The relation of the maximum velocity $v_{\max}$ to the inflow and outflow
	parameters $a$ and $b$ defines different stationary regimes, in which boundary layers arise at the entrance and/or the exit.\\
	The following regimes were introduced in~\cite{Wood2009} and were characterized and analyzed by Burger and Pietschmann in \cite{Burger2016}:
	
	\begin{enumerate}[label=(\arabic*),leftmargin=25pt,parsep=0pt]
		
		\item \textit{Influx limited phase} in the case $a < b$ and $\min\left(a,b\right)<\frac{v_{\max}}{2}$. We observe an asymptotically low density 
		($\rho < \frac{1}{2}$) and a boundary layer at the exit.
		
		\item \textit{Outflux limited phase} for $a > b$ and $\min\left(a,b\right)<\frac{v_{\max}}{2}$. This creates an 
		asymptotically high density ($\rho > \frac{1}{2}$) and a boundary layer at the entrance.
		
		\item \textit{Maximal current phase} if $a, b\geq \frac{v_{\max}}{2}$. Here we 
		observe an asymptotic density with $\rho  \approx \frac{1}{2}$ in most of 
		the domain when $\sigma$ is small and 
		boundary layers at the entrance and exit. In the case $a = b = \frac{v_{\max}}{2}$,  the unique steady state is given by $\rho(x) = \frac{1}{2}$.
	\end{enumerate}
	\begin{figure}
		\begin{center}
			\includegraphics[width=0.3\textwidth]{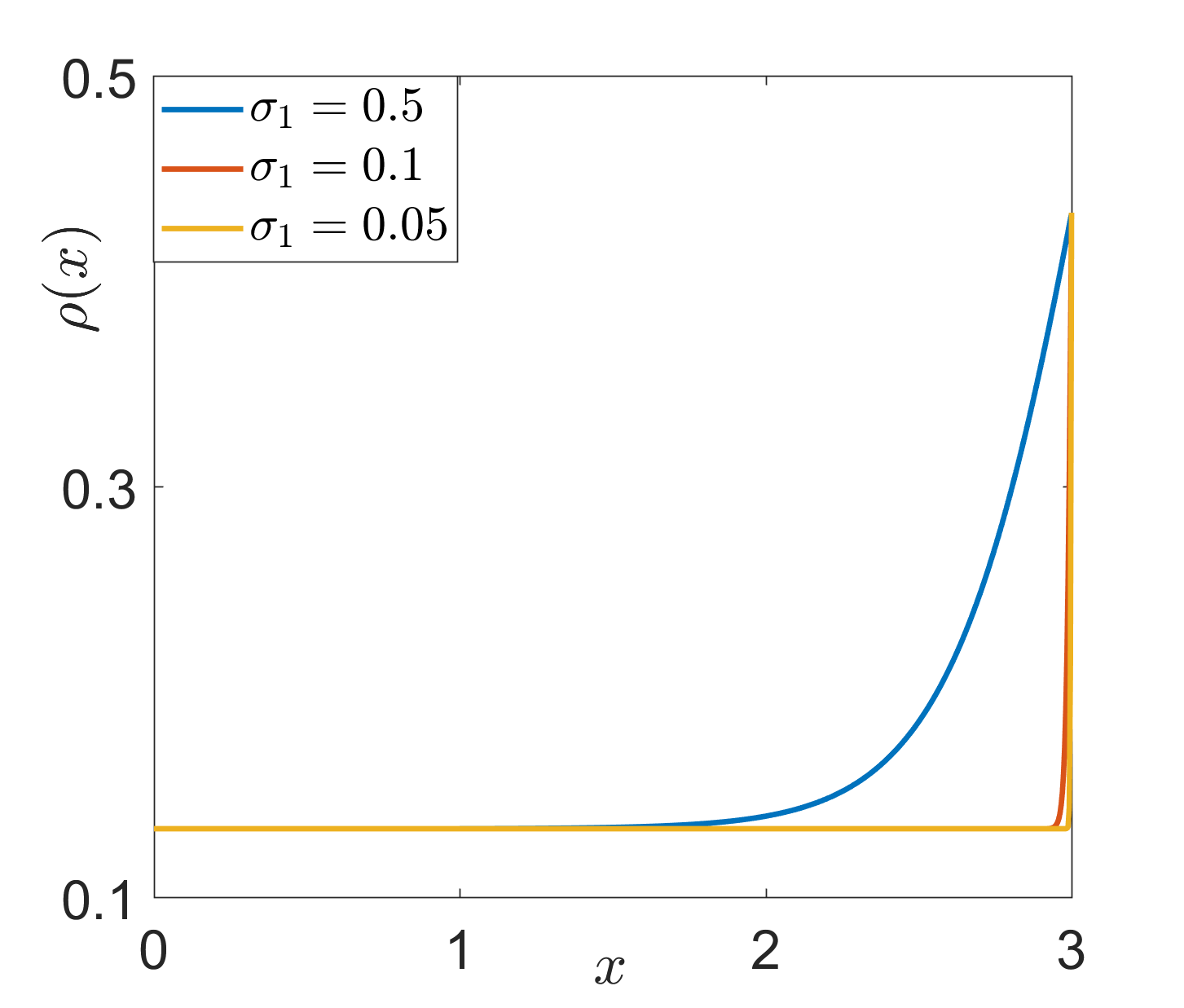}
			\includegraphics[width=0.3\textwidth]{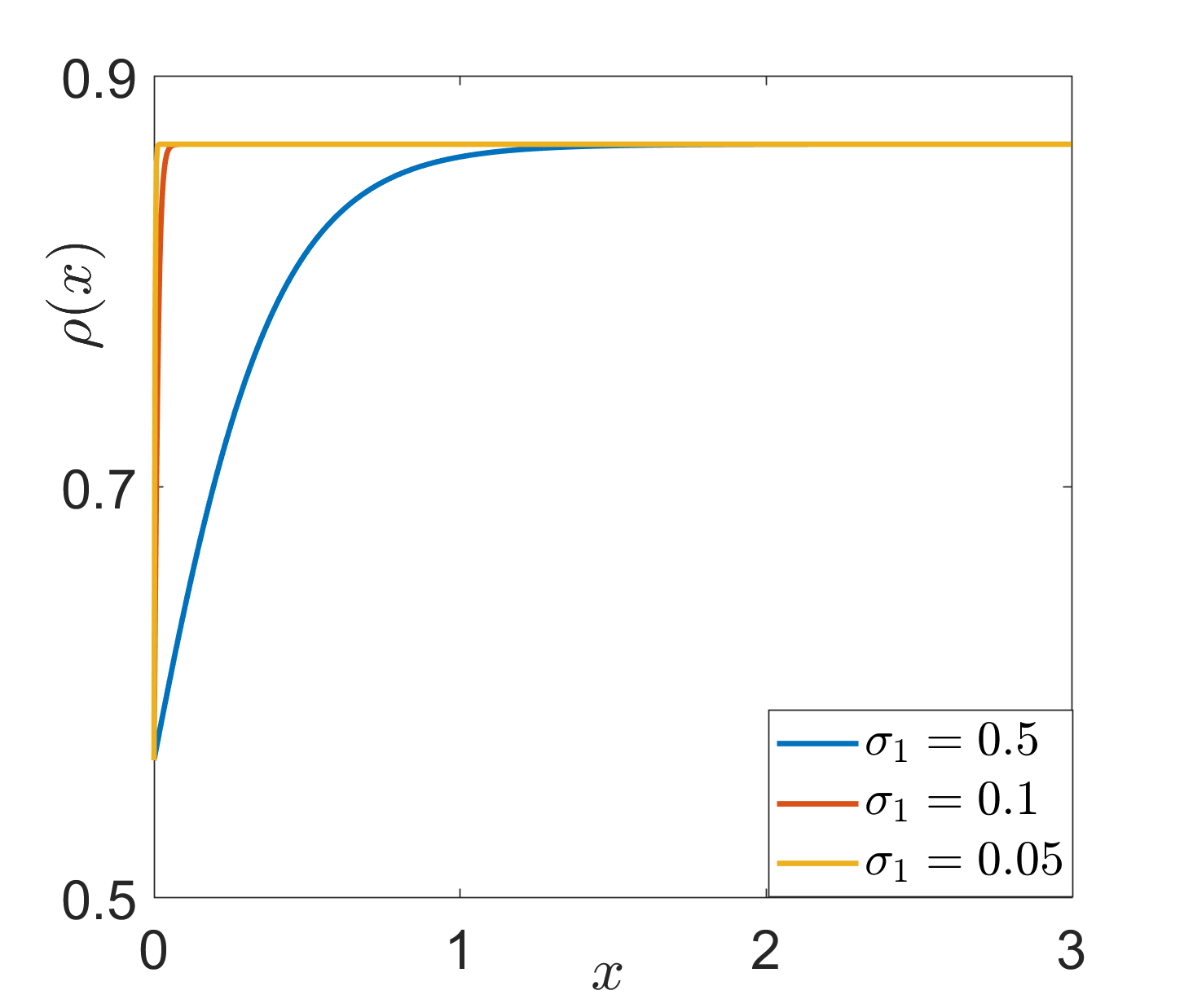}
			\includegraphics[width=0.3\textwidth]{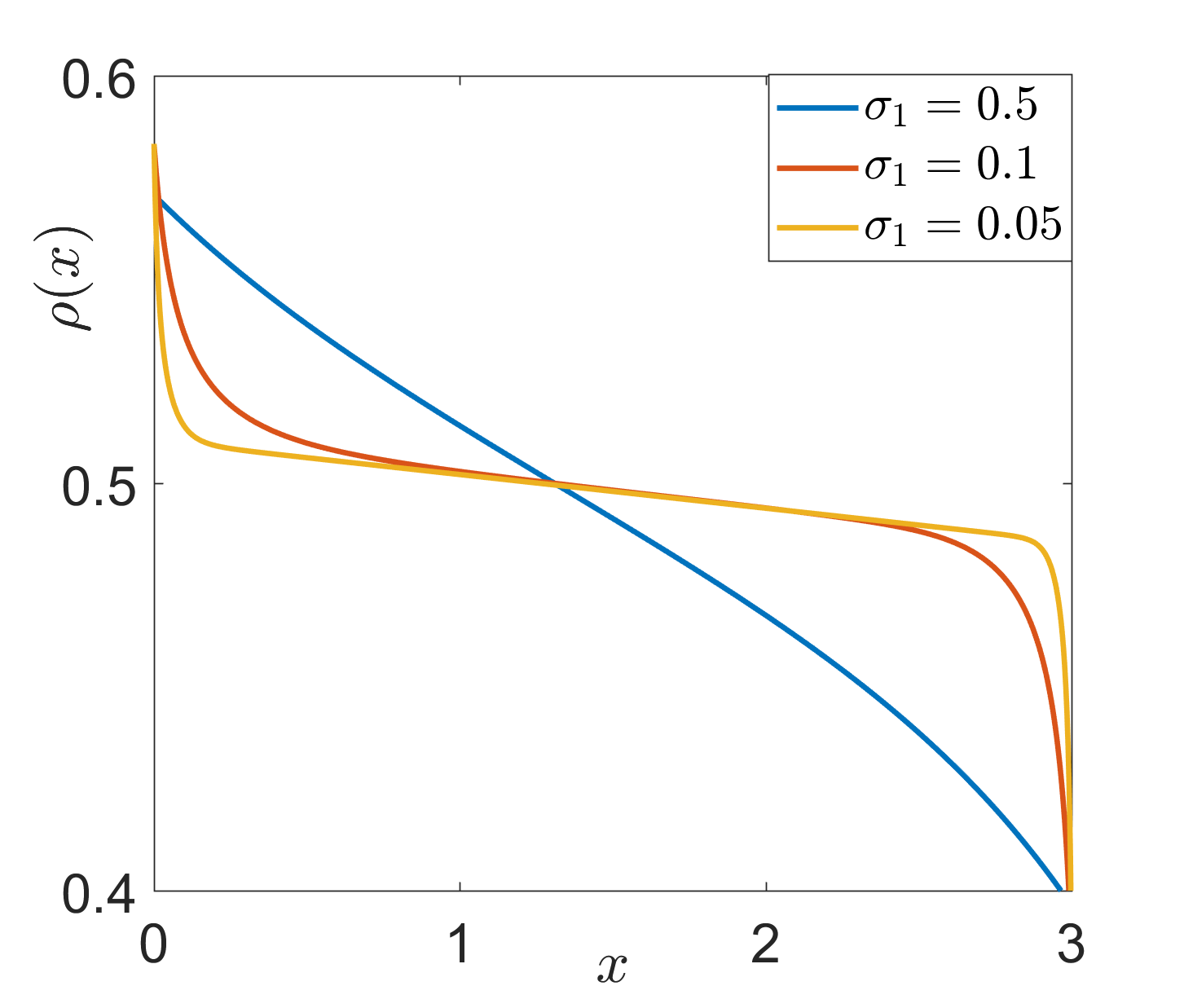}
			\caption{Examples of the steady states of the Fokker--Planck equation in one dimension for $\sigma = 0.5, \, 0.1, \, 0.05$ and (left) the influx limited phase
				-- $a = 0.2,\, b = 0.4$, (center) the outflux limited phase -- $a= 0.4,~b = 0.2$, and (right) the maximal current phase -- $a = 0.9,\, b = 0.975$. }\label{f:stationary_regimes}
		\end{center}
	\end{figure}
	
	Figure \ref{f:stationary_regimes} illustrates the different regimes in 1D in the case of a corridor of length $L=3$, $v_{\max}=1.5$ and different values of $\sigma = \sigma_1$. We observe the predicted
	boundary layers at the entrance and/or exit, whose width is determined by the diffusivity $\sigma$. The existence of at least one steady state solution under suitable assumptions on the in- and outflow boundary conditions as well as a more
	general regular potential was proven in \cite{Burger2016}. Note that these steady state profiles are unique in 1D.
	
	In the following we state and discuss the respective existence result for the time dependent problem. 
	The proof uses similar techniques as in \cite{Burger2016} and can be found in the Appendix.
	Without loss of generality we assume $\sigma_1 = \sigma_2 = 1$ and 
	$v_{\max} = 1$. 
	Furthermore, we consider the more general case in which $e_1$ is replaced by $\nabla V$ in \eqref{e:F1positive}. We make the following assumptions:
	\begin{enumerate}[label=(A\arabic*)]
		\item $\Omega \subset \mathbb{R}^n$ bounded, $n=1,2,3$, with boundary $\partial \Omega \in C^2$. \label{a:domain}
		\item $0 \leq a, b \leq 1$, and $\min\lbrace a, b \rbrace < 1$. \label{a:alphabeta}
		\item $V \in W^{1,\infty}(\Omega)$ and $\Delta V = 0$. \label{a:potential}
	\end{enumerate}
	Assumption \ref{a:domain} actually excludes domains like the corridor shown in Figure~\ref{f:sketch}. We still employ the boundary conditions 
	\eqref{eq:BC-inflow-rescaled}--\eqref{eq:BC-vertical-rescaled}, simply assuming that the union of (the disjoint)
	$\Gamma_{in}, \Gamma_{out}$ and $\Gamma_N$
	comprise the entire boundary $\partial \Omega.$ 
	One could consider a rectangular domain with rounded off corners instead,
	but since we did not observe any issues in the computational experiments with
	corners, we retained them for a more realistic simulation setting. 
	
	We seek weak solutions in the set
	\begin{align}\label{e:solset}
		\mathcal{S} = \lbrace \rho \in \mathbb{R}: ~ 0 \leq \rho \leq 1 \rbrace,
	\end{align}
	and refer to its interior as $\mathcal{S}^0$. Then the weak solution $\rho: \Omega \times [0,T] \rightarrow \mathcal{S}$ satisfies
	\begin{align}\label{eq:weakFP}
		\int_0^T \left[\langle \partial_t \rho, \varphi \rangle_{H^{-1}, H^1} - \int_{\Omega} j \cdot \nabla \varphi~ dx - a \int_{\Gamma_{in}} (1-\rho) \varphi ~ds + b \int_{\Gamma_{out}} \rho \varphi~ ds \right]dt = 0,
	\end{align}
	for all test functions $\varphi \in L^2(0,T;H^1(\Omega))$. 
	\begin{theorem}\label{t:globalexistence}
		Let assumptions \ref{a:domain}-\ref{a:potential} be satisfied. Let $\rho_0: \Omega \rightarrow \mathcal{S}^0$ be a measurable function such that {the initial entropy $E$ defined by \eqref{e:entropy}} satisfies $E(\rho_0) < \infty$. Then there exists a 
		weak solution to system \eqref{e:system-rescaled}  in the sense of \eqref{eq:weakFP} satisfying
		\begin{align*}
			\partial_t \rho \in L^2(0,T; H^1(\Omega)^*)\\
			\rho \in L^2(0,T; H^1(\Omega)).
		\end{align*}
	\end{theorem}

	In the preceding theorem, the star denotes the dual operation.
	The existence proof as well as improved regularity results can be found in the Appendix.

	\subsection{Analytic considerations on the generalized McKean--Vlasov equation}
	The analytic results for the Fokker--Planck equation~\eqref{e:system-rescaled} would allow us to prove existence and uniqueness of strong solutions for~\eqref{eq:sde} 
	in $\mathbb{R}^2$. However, since we consider the setup of a corridor with a mix of boundary conditions (reflected in the walls, and partially reflected in the entrance and exit), 
	standard results are not applicable. The only available results for SDEs with boundaries are slight generalizations of Skorokhod's problem~\cite{Skorokhod} and can only be used for 
	SDEs in the half plane. We are not aware of any results for more complicated geometries. However, our numerical experiments show that a careful discretization 
	of the process, and in particular treatment of trajectories impinging on
	the boundary of $\Omega$, yields bounded and continuous solutions of equation~\eqref{eq:sde}.

	\section{Parameter estimation}
	\label{sec:estimation}
	In this section we introduce the  parameter estimation framework for SDEs which depend on their density. We will use a Bayesian approach based on sampling from a posterior distribution, and confirm our results with the 
	computation of a maximum a posteriori (MAP) estimator.  We recall that our initial goal was to estimate both $\rho_{\max}$ and $v_{\max}$. However, as we have seen, determination of $\rho_{\max}$ is not possible using the current setting and only
	determination of $v_{\max}$ is feasible.
	
	We estimate $v:=v_{\max}$ from a collection of sample paths $\{X_i(t)\}_{t \in [0,T]}^{i=1,\dots,J}$ which are realizations of the McKean--Vlasov equation~\eqref{eq:sde}. 
	In what follows we parameterize  $f(\cdot)$ by $v$ and write $f(\cdot;v)$. For ease of presentation, we discuss the estimation from a single trajectory before generalizing it for multiple trajectories. 
	
	Let $X=X(t)$ be a realization of \eqref{eq:sde}. Throughout the remainder
	of the paper we use the notation
	$$|\cdot|_{A}=|A^{-\frac12}\cdot|$$
	where $|\cdot|$ is the Euclidean norm and $A$ any positive-definite
	symmetric matrix; a corresponding inner-product may be defined by
	polarization. Since ${\dot{W}}$ is white, it is intuitive that
	finding the best value of the parameter $v$, given an observation of a trajectory $X$, corresponds to minimizing the function
	\begin{equation}
		\label{eq:phi}
		\Phi(v;X)= \frac14\int_0^T |\dot{X}-F\bigl(\rho(X(t),t);v\bigr)|^2_{\Sigma},
	\end{equation}
	over all possible values of $v$. However, the function $\Phi(v;\cdot)$ is almost surely infinite. In order to avoid this problem, we note that we can write
	\begin{equation}
		\label{eq:phi-to-psi}
		\Phi(v;X)=\Psi(v;X) +\frac{1}{4} \int_0^T |\dot{X}|^2 \ dt,
	\end{equation}
	where the function $\Psi$ is defined as
	\begin{equation}
		\label{eq:psi}
		\Psi(v;X)=\frac14\int_0^T \Bigl(|F\bigl(\rho(X(t),t);v\bigr)|^2_{\Sigma} dt-
		2\langle F\bigl(\rho(X(t),t);v\bigr), dX(t) \rangle_{\Sigma}\Bigr).
	\end{equation}
	Note that the last term in $\Psi$ is an It\^o stochastic integral. 
	It is preferable to perform the parameter estimation using $\Psi$ 
	instead of $\Phi$, since the latter is infinite almost surely, whilst the
	former is finite almost surely, and since $\int_0^T|\dot{X}|^2 \ dt $ does not depend explicitly on $v$ both 
	cases yield the same results when applied in discretized form.
	
	We will also add prior information and
	perform the parametric estimation of $v_{\max}$ by minimizing 
	\begin{equation}\label{eq:J}
		\mathcal{J}(v;X):=\Psi(v;X)+\frac{1}{2c}|v-m|^2,
	\end{equation}
	subject to the constraint that $v$ is positive. {The parameters $m$ and $c$
		correspond, respectively, to a prior estimate of the velocity and the variance
		associated with this estimate. This optimization problem will now be derived
		through a detailed description of the Bayesian formulation of inversion. 
		Specifically we define precisely the Bayesian problem for the posterior distribution
		of $v$ given a trajectory $\{X(t)\}_{t \in [0,T]}$; this posterior distribution
		is maximized at the constrained minimization problem just defined.}  
	
	In this Bayesian context 
	the functional $\mathcal{J}$ can be interpreted as follows.
	The first term measures the misfit between the observed 
	trajectory and the predicted density regime obtained from the FP equation. 
	The second term is a regularization term, which is weighted by prior knowledge on $v$. In particular, we assume that $v$ is normally distributed with mean $m$ and variance $c$, conditioned to be positive.
	The probabilistic interpretation of $\mathcal{J}$ is as follows. The 
	function $\exp(-\mathcal{J}(v;X))\mathbb{1}(v>0)$, appropriately normalized, is a probability distribution and in fact the posterior distribution, $\mathbb{P}(v|X)$, of $v$ given a realization $X$. This conditional distribution can be obtained, up to a normalization constant which is independent of $v$,
	by computing the joint probability of the process $X$ and the random variable $v$, which is given by
	\begin{equation*}
		\mathbb{P}(X,v) = \mathbb{P}(v|X)\mathbb{P}(X);
	\end{equation*}
	thus
	\begin{equation}\label{e:joint}
		\mathbb{P}(v|X) \propto \mathbb{P}(X,v),
	\end{equation}
	where the constant of proportionality is independent of $v$.
	Consider the probability measure $\mathbb{\bblambda} = \mathbb{Q}\otimes\operatorname{Leb}(\R)$, where $\mathbb{Q}$ is the law of the Brownian motion driving the process and $\operatorname{Leb}(\R)$ is the Lebesgue measure in $\R$. Lemma 5.3 in \cite{HSVP07} states that 
	\begin{equation}\label{e:jointderivative}
		\frac{d\mathbb{P}}{d\bblambda} (X,v)= \frac{d\mathbb{P}}{d\mathbb{Q}}(X|v)\pi_0(v),
	\end{equation}
	where $\pi_0$ is the prior Lebesgue density for $v$. 
	Using Girsanov's Theorem~\cite{Elworthy1982,Oksendal1998}, we see that
	\begin{equation}\label{eq:Girsanov-sde}
		\frac{d\mathbb{P}}{d\mathbb{Q}}(X|v) = \exp(-\Psi(v;X)),
	\end{equation}
	where $\Psi$ is defined in~\eqref{eq:psi}. If we choose $\pi_0(dv) = \mathbb{1}(v>0)N(m,c)(dv)$, which ensures that the velocity is almost surely
	positive both {\em a priori} and hence {\em a posteriori}, 
	equation \eqref{e:joint} gives
	\begin{equation}
		\label{e:p}
		\mathbb{P}(v|X) \propto \exp\Bigl(-\Psi(v;X) - \frac{1}{2c}|v-m|^2\Bigr)\mathbb{1}(v>0),
	\end{equation}
	which is exactly $\exp(-\mathcal{J}(v;X))\mathbb{1}(v>0)$.
	
	We can either sample from the probability distribution
	given by \eqref{e:p} (the Bayesian approach) or we can simply
	minimize the negative log likelihood of the data $\mathcal{J}(v;X)$, 
	subject to $v$ being positive.
	While the fully Bayesian framework allows us to quantify the uncertainty on our estimate, the minimization approach only gives 
	a single point estimator. It is however, much cheaper to simply
	minimize.  In the following sections we will compare the results of the 
	two approaches by sampling from the posterior distribution using 
	the pCN algorithm, and minimizing the function $\mathcal{J}$ using the Nelder--Mead algorithm. Both methods are derivative free, so they do not involve any extra function evaluation due to computing derivatives. We 
	note that other derivative free samplers or optimizers could be used.
	
	The accuracy of the aforementioned estimators can be expected to
	improve with the lifetime of the observations \cite{prakasarao}.
	Since in our setting  trajectories terminate when they exit the domain, we are only able to observe them for a finite time. 
	We will see that using multiple trajectories to estimate parameters
	has similar desirable properties to observing a single trajectory
	of a standard (boundaryless) SDE over a long time horizon.
	
	We denote a family of such trajectories by $X_j(t), \, j=1,\dots,J$. 
	We assume that each trajectory $X_j$ starts at $t_0^j$ with $X_j(t_0^j) = X_j^0$ and exits at time $t = t_f^j$. 
	We set $F(\rho(X_j(t),t);v_{\max}) = 0$ for $t\notin[t_0^j,t_f^j]$, which allows us to define
	\[
	\psi_j(v;X_j(t)) = \frac{1}{4}\int_0^T \Bigl(|F\bigl(\rho(X_j(t),t);v\bigr)|^2_{\Sigma} \ dt-
	2\langle F\bigl(\rho(X_j(t),t);v\bigr), dX_j(t) \rangle_{\Sigma}\Bigr).
	\]
	Then we can sum the misfit of all trajectories and obtain
	\begin{equation}\label{eq:Psi_mult_traj}
		\Psi_J\left(v;X_1,\dots,X_J\right) = \sum_{j=1}^J \psi_j(v;X_j).
	\end{equation}
	Since all $\psi_j$ are independent, we can apply the previous argument, replacing $\Psi(v,X)$ by \eqref{eq:Psi_mult_traj}. The averaging over all the trajectories has the same effect as considering large values of $T$ for a single trajectory. This can be shown formally, using techniques similar to those used in~\cite{Kutoyants2004} for a single trajectory.
	
	It could be of interest to also estimate the diffusion coefficient matrix $\Sigma$. Estimating constant diffusion coefficients for sufficiently frequent 
	observations is, in theory, a well-understood problem, 
	see \cite{Iacus2009,prakasarao}. 
	However our model is likely to be inconsistent with the data at small scales
	and so it is important to appreciate that estimation of $\Sigma$ might
	well be a non-trivial problem, requiring techniques such as those
	introduced in \cite{pavliotis2007parameter}.
	%
	%
	\section{Computational methods}
	\label{sec:comp}
	In this section we discuss the discretization of the nonlinear PDE \eqref{e:system-rescaled} and SDE 
	\eqref{eq:sde} before continuing with parameter estimation in the
	following section. Note that the parameter estimation is computationally 
	costly in the sense that it involves the solution of a PDE in every 
	iteration of the sampling and optimization algorithms. Motivated by 
	relevant experimental settings (noting that data collection starts after the experiment is equilibrated), we explore the effect of using both
	stationary and transient density regimes on the quality of our estimates. Using steady state density profiles also has the advantage of a lower computational cost. However we will show that the steady state regimes do not produce consistent estimates across all parameter regimes, an effect which is mitigated
	by using (the more expensive) time-dependent density based estimation.

	\subsection{The time dependent and steady state Fokker--Planck solvers}
	
	The time dependent solver is based on the decomposition of \eqref{e:FP-rescaled} into a diffusive and a convective part. 
	For the nonlinear convective operator an upwind discontinuous Galerkin (DG) method and an explicit time integration scheme 
	is used. For the discretization of the linear diffusive part we use a hybrid discontinuous Galerkin (HDG) method and an 
	implicit time discretization. This additive time splitting, also known as Implicit-Explicit (IMEX) method, allows to 
	treat the stiff diffusive term implicitly while the nonlinear non-stiff hyperbolic problem is solved explicitly, see \cite{ascher1995implicit}.
	The IMEX scheme allows for larger time steps, while the (H)DG discretization ensures stability and parallelizablilty.

	In some cases we employ
	1D steady state profiles $\rho_s = \rho_s(x)$ (extended as constants in the
	orthogonal direction) and these are calculated using an $H^1$ conforming 
	finite element discretization. We use a damped Newton method
	to solve the resulting nonlinear system, with an initial guess depending on the relation of $a$ and $b$. Both solvers use the finite element software package Netgen/Ngsolve, see \cite{ngsolve}.	
	
	\subsection{SDE solver and trajectory generation}\label{s:SDE}
	
	We approximate all trajectories of \eqref{eq:sde}
	by using the explicit Euler-Maruyama scheme~\cite{Higham2001}, which is a forward Euler time-stepping method. It consists of 
	defining the approximation $X_k \approx  X(t_k): = X(k\Delta t)$ and updating the individual position as
	\begin{equation}\label{eq:E-M}
		X_{k+1} = X_k + F(\rho(X_k,t_k)) \Delta t + \sqrt{2\Sigma\Delta t}\xi_k,
	\end{equation}
	where $\xi_k\sim N(0,1)$ and $\rho$ is the solution to the Fokker--Planck equation (which we advance in time simultaneously). 
	To generate trajectories using the steady state profiles, we replace $\rho(X_k,t_k)$ by the steady state solution $\rho_s(X_k)$. 
	
	Since the diffusion coefficient is constant, the Euler-Maruyama scheme has strong order of convergence one~\cite{Higham2001}, and we do not need to use more complicated time-stepping methods. 
	At the boundaries we implement individual rules consistent with ~\eqref{e:system-rescaled}. Singer \emph{et al}~\cite{Singer2008} and Erban and Chapman~\cite{Erban2007} derived appropriate
	scaled rates of entrance and exit for partially reflected boundaries in the half plane. These ensure that the Robin boundary conditions are satisfied by the corresponding probability densities. Even though it is not established 
	that their results are directly applicable in the case of the more 
	complex geometries we use here, we nonetheless follow their approach, 
	and implement the entrance and exit behavior as 
	described in Table~\ref{tab:bcs}.
	\begin{table}
		\begin{center}
		\begin{tabular}{ | M{2.75cm} | M{2.75cm} | M{2.75cm} | M{3cm} |}
			\hline
			Boundary &  PDE &  SDE & Implementation 
			\\ 
			\hline
			\begin{center}$\Gamma_N$\end{center} & \begin{center}Neumann \end{center} & \begin{center}Reflecting\end{center} & \begin{center}
				\vskip0.2cm
				\begin{tikzpicture}[scale=0.75]
				\draw[thick] (1,0) -- (4,0) ;
				\fill (1.9,-0.6) circle (3pt);
				\draw (3.1,-0.6) circle (3pt);
				\draw[thick,-] (2.0,-0.5) -- (2.5,0);
				\draw[thick,->] (2.5,0) -- (3.05,-0.55)node [below,xshift=0.5cm] {{\tiny $p = 1$}};
				\end{tikzpicture}\end{center}
			\\
			\begin{center}$\begin{array}{c}\Gin\\ \text{(from outside)}\end{array}$ \end{center} & \begin{center}Robin\end{center} & \begin{center}Wait to get in\end{center} & 
			\begin{center}\begin{tikzpicture}[scale=0.75]
				\draw[thick] (1,-1.5) -- (1,0.25) ;
				\draw[dashed,->] (0.5,-0.5)  -- (2.0,-0.5)  node [above,yshift=-0.5cm] {{\tiny $p = P_{in}$}};
				\draw[dashed,->] (0.5,-0.5) .. controls +(down:10mm) and +(left:10mm) ..  (0.4,-0.450) node [below,yshift=-0.3cm,xshift=-0.2cm] {{\tiny $p = 1-P_{in}$}};
				\fill (0.5,-0.6) circle (3pt);
				\draw (2.1,-0.5) circle (3pt);
				\draw (0.5,-0.4) circle (3pt);
				\end{tikzpicture}\end{center}\\
			\begin{center}$\begin{array}{c}\Gin\\ \text{(from inside)}\end{array}$\end{center}& \begin{center}Robin \end{center}& \begin{center}Partially reflecting \end{center}& 
			\begin{center}\begin{tikzpicture}[scale=0.75]
				\draw[thick] (1,-0.75) -- (1,0.75) ;
				\draw[thick] (1.8,-0.65)  -- (1.0,0);
				\draw[dashed,->] (1.0,0) -- (1.70,0.60) node [above,xshift=0.2cm] {{\tiny $p = 1-P_{in}$}};
				\draw[dashed,->] (1.0,0) -- (0.7,0.0) node [above,xshift=-0.5cm] {{\tiny $p = P_{in}$}} node[below,xshift=-0.5cm]{{\tiny back to $\Gin$}};
				\fill (1.8,-0.65) circle (3pt);
				\draw (1.8,0.65) circle (3pt);
				\draw (0.65,0) circle (3pt);
				\end{tikzpicture}\end{center}
			\\
			\begin{center}$\Gout$\end{center} & \begin{center}Robin\end{center} & \begin{center}Partially reflecting\end{center} & \begin{center}\begin{tikzpicture}[scale=0.75]
				\draw[thick] (1,-0.75) -- (1,0.75) ;
				\draw[thick] (0.2,-0.65)  -- (1.0,0);
				\draw[dashed,->] (1.0,0) -- (0.2,0.60) node [above,xshift=-0.2cm] {{\tiny $p = 1-P_{out}$}};
				\draw[dashed,->] (1.0,0) -- (1.4,0.0) node [above,xshift=0.5cm] {{\tiny $p = P_{out}$}} node[below,xshift=0.6cm] {{\tiny leave domain}};
				\fill (0.2,-0.65) circle (3pt);
				\draw (0.2,0.65) circle (3pt);
				\draw (1.5,0) circle (3pt);
				\end{tikzpicture}\end{center}
			\\
			\hline
		\end{tabular}
	\end{center}
		\caption{Implementation of the boundary conditions at the SDE level. The filled dot represents the current state, while the empty circles represent the possible position in the following time step, which is realized with probability $p$. Here we use  the rescalings $P_{in} = \sqrt{\pi\Delta t /(2\sigma_1^2)}a(1-\rho(0,x_2))$ and $P_{out} = \sqrt{\pi\Delta t /\sigma_1^2}b\rho(L,x_2)$ according to \cite{Erban2007}.}\label{tab:bcs}
	\end{table}
	
	\subsection{Derivative-free optimization -- the Nelder--Mead algorithm}
	As detailed in Section~\ref{sec:estimation}, we will use derivative-free 
	methods for the minimization of \eqref{eq:J}, corresponding to
	finding the MAP estimator of the posterior probability distribution: the
	maximizer of \eqref{e:p}. 
	We choose the Nelder--Mead algorithm~\cite{NelderMead}, which is a widely used nonlinear unconstrained optimization algorithm. It is a direct search method which is based on the function evaluation of $n+1$ points in an $n-$dimensional space.
	A rough outline of the method (see~\cite{Conn2009} for more details) is given in Algorithm~\ref{alg:NM}. The number in boldface is the total number of function evaluations necessary if that particular step is performed. 
	\begin{algorithm}
		\caption{Nelder--Mead algorithm}
		\label{alg:NM}
		\begin{algorithmic}[1]
			\STATE{Choose an initial simplex $\{v^0,\dots,v^n\}$ and compute $\mathcal{J}^i = \mathcal{J}(v^i;X)$.}
			\WHILE{Diameter of simplex $> $ tolerance,}
			\STATE{{\bf Order} the $n+1$ vertices so that $\mathcal{J}^0\leq \mathcal{J}^1\leq\cdots\leq \mathcal{J}^n$.}
			\STATE{{\bf Reflect:} reflect the worst vertex $v^n$ over the centroid of the simplex, obtaining $v^r$. If $\mathcal{J}^0\leq \mathcal{J}^r<\mathcal{J}^{n-1}$, replace $v^n$ by $v^r$. {\bf [1]}}
			\STATE{If $\mathcal{J}^r<\mathcal{J}^0$, {\bf Expand} the reflected point, obtaining $v^e$.\\ If $\mathcal{J}^e\leq \mathcal{J}^r$,replace $v^n$ by $v^e$. Otherwise replace $v^n$ by $v^r$. {\bf [2]}}
			\STATE{If $\mathcal{J}^r\geq \mathcal{J}^{n-1}$, {\bf Contract:} the reflected point, obtaining $v^c$. \\ If $v^c$ improves on $\mathcal{J}^n$ then replace $v^n$ by $v^c$.{\bf [2]}}
			\STATE{Otherwise, perform a {\bf Shrink}. {\bf [n+1]}}
			\ENDWHILE
		\end{algorithmic}
	\end{algorithm}
	If the objective function is strictly convex then the algorithm only requires 1 or 2 function evaluations per iteration. Furthermore, it converges in 1D for strictly convex functions~\cite{Lagarias1998}. In 2D only limited convergence results are known. 
	We use \textsc{Matlab}'s implementation of the Nelder--Mead algorithm (via the function \emph{fminsearch}) with a modification of $\mathcal{J}$ that penalizes negative values of $v$. We observe a large decrease of the objective function in the first few iterations, as reported in the literature for many other
	applications.

	\subsection{MCMC method -- the pCN algorithm}
	The MAP estimator obtained with the Nelder--Mead algorithm does not provide us with any quantification of uncertainty on our estimates. In order
	to obtain such uncertainty estimates we sample from the posterior distribution
	given by \eqref{e:p}. 
	For this purpose we use the pCN (pre-conditioned Crank--Nicholson) method 
	described in~\cite{CRSW08}; it is a general purpose method for sampling
	from any distribution which is formed as the product of a (not
	necessarily smooth) function $\mathfrak{f}$ and a Gaussian density and
	has the advantage that derivatives of $\mathfrak{f}$ are not required. 
	In our problem this means that we avoid computing derivatives of $\Psi$.
	Looking forward, the method also has the potential to
	scaling up to cases where the objective is to estimate the shape of  
	$f(\rho)$ in \eqref{eq:F} nonparametrically.
	Furthermore, the algorithm has a single tuneable parameter $\beta$ which can be explored to maximize efficiency. Algorithm ~\ref{alg:pCN} gives the details. 
	
	\begin{algorithm}
		\caption{The pCN algorithm}
		\label{alg:pCN}
		\begin{algorithmic}[1]
			\STATE{Set $k=0$ and pick $v^{(0)}$.}
			\FOR{$k=1,...,N$, where $N$ is the number of iterations,}
			\STATE\label{line3}{Propose $y^{(k)}=m
				+\sqrt{(1-\beta^2)}(v^{(k)}-m)+\beta \xi^{(k)},
				\quad \xi^{(k)} \sim N(0,c).$}
			\STATE{ Set $v^{(k+1)}=y^{(k)}$ with probability $\alpha_k:=\alpha(v^{(k)},y^{(k)})$}
			\STATE{Set $v^{(k+1)}=v^{(k)}$ otherwise.}
			\STATE{$k \to k+1$.}
			\ENDFOR
		\end{algorithmic}
	\end{algorithm}
	
	The acceptance probability is given by
	\begin{equation}
		\label{eq:accp}
		\alpha(v,y)=\min\{1,\exp\bigl(\Psi(v;X)-\Psi(y;X)\bigr)\}\mathbb{1}(y>0).
	\end{equation}
	Notice, in particular, that if $\Psi(y^{(k)};X) \le \Psi(v^{(k)};X)$
	and $y^{(k)}>0$ then the proposed move is accepted with probability one.

	\section{Numerical results}
	\label{sec:numerical}

	In this section we present and discuss numerical results based on the solvers and methodologies presented in Section \ref{sec:comp}. In particular
	we want to estimate $v_{\max}$ using multiple trajectories of the coupled
	SDE-FP system. These trajectories are generated using both the time-dependent 
	and the steady state solution of the FPE, and for different inflow and outflow 
	conditions. The corresponding parameter estimation problem is then solved using
	the Nelder--Mead optimizer or the pCN Bayesian sampler. 
	Our numerical simulations lead to the following conclusions: 
	\begin{enumerate}
		\item the maximum speed $v_{\max}$ is learnable; { see Figure \ref{f:time-step}}
		\item results are more accurate and less uncertain if the influx and outflux rate $a$ and $b$ differ considerably; {see Figure \ref{f:traj_TD}}
		\item we do not obtain reliable estimates for all regimes using stationary density profiles; { see Figure \ref{f:PM_S}}
		\item estimates obtained from trajectories experiencing steady state densities are less uncertain than from time dependent ones (in the regimes where
		steady state estimation does give consistent results); { see Figures \ref{f:traj_TD} and \ref{f:traj_S}}.
	\end{enumerate} 
	
	In order to test the parameter estimation methodology, we generate a collection of trajectories for a range of
	parameters which represent the different density regimes. To create this data
	we choose the following
	parameters for all trajectories: corridor length $L = 3$, corridor width $2\ell = 0.5$, noise strength $\sigma_1 = \sigma_2 = \sigma_0 = 0.05$
	and a final time $T=2$. The time steps in the SDE solver are set to $\Delta t_{SDE} = 10^{-3}$.
	We assume that the true value for the maximum velocity is $v_{\max} = 1.5$ m/s, which is chosen close to values obtained from experiments. The final time is chosen sufficiently large to ensure that some 
	individuals have the time to leave the domain. The size of the time steps corresponds to the setting in a high frequency regime.
	We generate trajectories using a time dependent density $\rho = \rho(x,t)$ satisfying \eqref{e:FP-rescaled} as well as the steady state $\rho_s = \rho_s(x)$. Note that
	the latter choice is computationally less expensive and corresponds more closely to widely used experimental conditions, since data is typically collected 
	once the pedestrian flow has equilibrated.
	We consider five sets of parameters which span the three steady state density regimes:
	\begin{enumerate}[label=-,leftmargin=20pt,parsep=0pt]
		\item \textit{{ Outflux} limited:} $a = 0.4, \, b = 0.2$ and $a = 0.45, \, b = 0.4$,
		\item \textit{{ Influx}  limited:} $a = 0.2, \, b = 0.4$ and $a = 0.1, \, b = 0.15$,
		\item \textit{Maximal current:} $a = 0.9, \, b = 0.975$.
	\end{enumerate}
	
	Since the variance coefficients might not be known in practice, we will estimate the parameters using values for the diffusion coefficient which differ
	significantly from those present in the data; we take
	$\sigma_1 = \sigma_2 = \sigma = 1$ in our algorithm.
	This corresponds to a form of {\em model-mispecification}
	and avoids committing an {\em inverse crime} \cite{KS05}.
	
	\subsection{Benchmarking}
	In the following, we present a first set of numerical results which confirm that the optimization and MCMC methodologies perform as expected. In particular, we will see that the value of the estimator for $v_{\max}$ is consistent across both methodologies, and is also independent of the various parameters used, such as time steps or spatial mesh, and the parameter $\beta$ of the pCN algorithm.
	Throughout this subsection we use $J=20$ trajectories.
	
	\paragraph{Influence of the discretization} 
	This test case is of particular interest with respect to the computational efficiency, since each iteration of either 
	minimization algorithm requires one (or more) PDE solves. Therefore we wish to use as large time steps as possible as 
	well as a mesh which is as coarse as possible.  We first present numerical results related to the influence of the PDE time 
	step on the value of the estimator for $v_{\max}$.
	\begin{figure}[h!]
		\includegraphics[width=\linewidth]{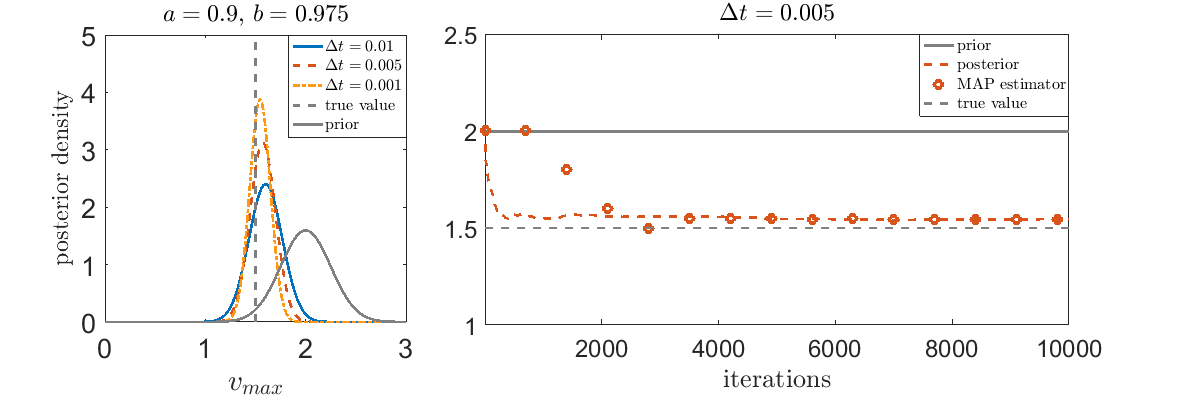}
		\caption{Influence of the PDE time step on the estimator for a maximal current regime with $a=0.9, b=0.975$. 
			Left: true value (gray dotted line), prior distribution (full gray line), and posterior distributions for different PDE time steps. 
			Right: prior mean (full gray line), true value (gray dotted line), and posterior mean (orange dotted line) and MAP estimator iterations (orange circles) 
			for $\Delta t = 0.005$.}\label{f:time-step}
	\end{figure}
	We run the pCN (and Nelder--Mead) algorithm using three different time steps for the PDE solve: $\Delta t = 0.001$, 
	$\Delta t = 0.005$ and $\Delta t = 0.01$ and interpolate the obtained values to fit the time step of the SDE trajectories. 
	We tested the influence of the time step in the maximal current regime with $a=0.9, \, b = 0.975$.
	The posterior distributions for $v_{\max}$ as well as the corresponding prior distribution and the true value are depicted 
	in Figure~\ref{f:time-step}. These results are consistent with those obtained using the Nelder--Mead algorithm to 
	compute the MAP estimator (successive iterations of the Nelder--Mead algorithm are depicted in the right panel with 
	orange circles, and we point out that the number of iterations for the MAP estimator was scaled by a factor of 
	approximately $10^{-3}$ for comparison).  We observe that all the test cases produce similar results, with a slightly 
	better agreement in the case of the finer time steps, where the results are also less uncertain, since the posterior 
	distribution has a smaller variance. For this reason, all the pCN and Nelder--Mead simulations presented below will 
	be performed with a PDE time step of $\Delta t = 0.005$.
	
	We observe a similar agreement between the mean of the posterior distribution and the MAP estimator for all the results 
	presented below, henceforth we present the posterior distribution only. We also point out that the MAP estimator (mode 
	of the posterior distribution) is approximately the same as the running average (mean of the posterior distribution). 
	This suggests that the posterior distribution is approximately
	Gaussian as is to be expected when the data is highly informative.
	Finally, we also tested the influence of the SDE time step (frequency of observations) and spatial mesh on the estimates, obtaining 
	similarly good results.
	
	\paragraph*{Influence of the parameter $\beta$} The value of the estimator obtained using the pCN algorithm should be independent of the parameter $\beta$ since
	$\beta$ is a tuneable parameter in the algorithm and not a parameter of the
	posterior distribution. In Figures~\ref{f:beta_TD} and~\ref{f:beta_S} we plot the prior and posterior distributions for $v_{\max}$ in two influx limited regimes, $a = 0.2, \, b = 0.4$ (left panels) and $a = 0.1, \, b = 0.15$ (right panels). Figure~\ref{f:beta_TD} was obtained in a time dependent regime, while Figure~\ref{f:beta_S} uses trajectories experiencing a steady state density.
	We observe in both figures that the posterior distribution is independent of the parameter $\beta$, as expected.

	\begin{figure}[h!]
		\includegraphics[width=\linewidth]{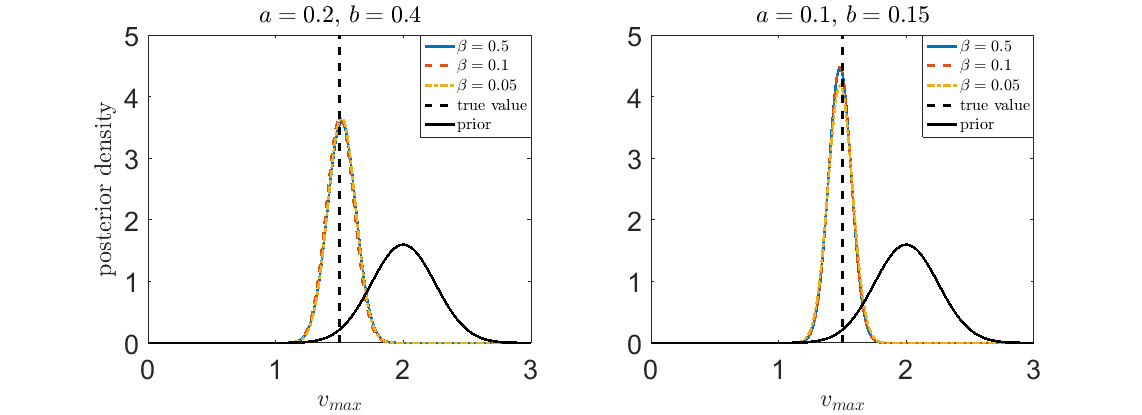}
		\caption{Influence of the parameter $\beta$ on the estimator for two influx limited regimes and using the solution of the time dependent Fokker--Planck equation. Left: $a=0.2, b=0.4$. Right: $a = 0.1, \, b = 0.15$.}\label{f:beta_TD}
	\end{figure}
	
	\begin{figure}[h!]
		\includegraphics[width=\linewidth]{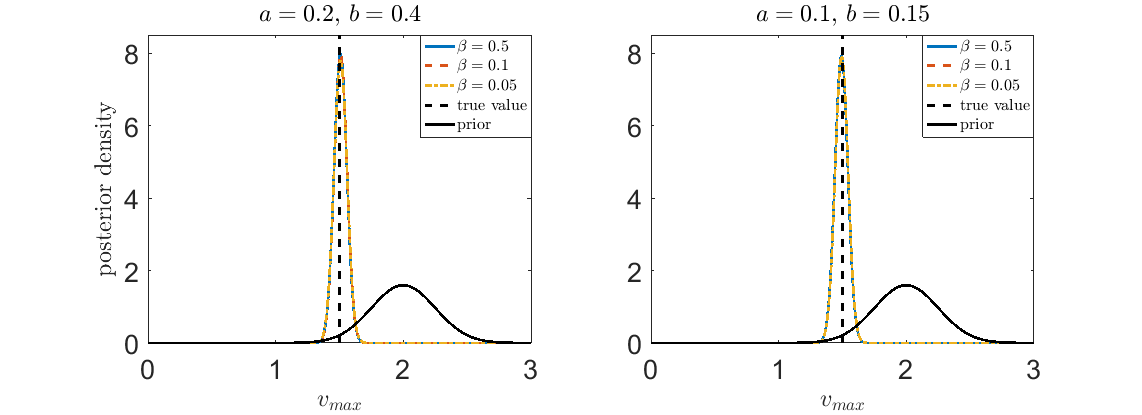}
		\caption{Influence of the parameter $\beta$ in the posterior distribution of $v_{\max}$ guesses for influx limited regimes and using the solution of the steady state Fokker--Planck equation. Left: $a=0.2, b=0.4$. Right: $a=0.1, b=0.15$.}\label{f:beta_S}
	\end{figure}

	\subsection{Influence of problem-specific parameters}
	In the following, we present our numerical results which confirm the conclusions presented at the beginning of this section.
	In particular, we will show that both approaches give consistent results with respect to the initial guess, prior mean and variance, and number of used trajectories. All the examples presented use $\beta = 0.1$ and $J=20$ trajectories, except the case in which we vary the number of trajectories.
	
	\paragraph*{Influence of the initial guess and prior mean} We assume that $v_{\max}$ has prior distributions $N(m,c)$ with $m\in\left\{1,2\right\}$ and $c = 0.25$ and test using 
	two initial guesses $v=1$ and $v=2$. Figure~\ref{f:PM_TD} depicts the prior distribution (gray full line), true value (gray dotted line) and the posterior distributions for each case tested. The MAP estimator agrees with the means of the posterior distributions for all cases. 
	\begin{figure}[h!]
		\includegraphics[width=\linewidth]{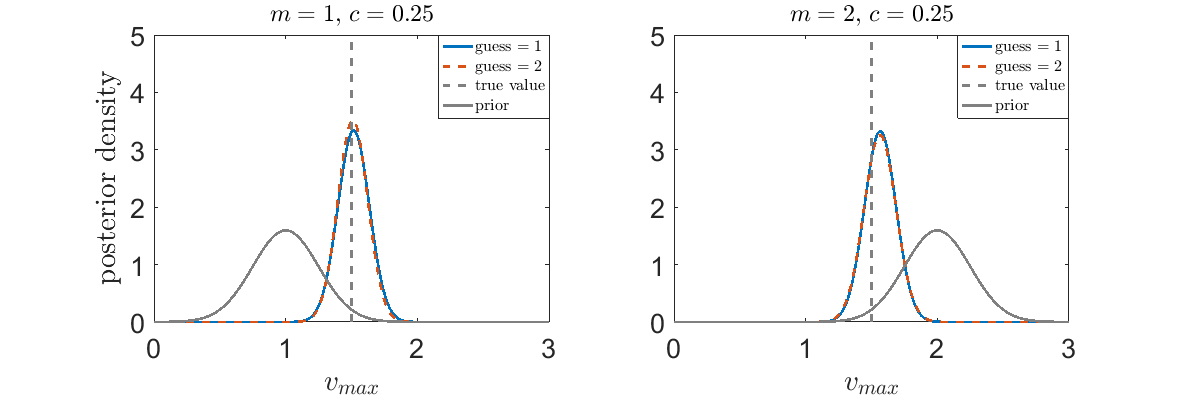}
		\caption{Influence of prior distribution and initial guess for an outflux limited regime with $a=0.45, b=0.4$  using solutions of the time dependent Fokker--Planck equation. 
			Left: Prior mean $m=1$. Right: Prior mean $m=2$. Both figures have $c=0.25$ and depict the prior distribution (full gray line), true value (dashed gray line) and posterior 
			distributions for initial guesses $v_{\max}=1, \, 2$. }\label{f:PM_TD}
	\end{figure}
	The left panel depicts the results for $m=1$ while the right panel has $m=2$. In each case we observe that the 
	estimates are independent of the initial guess, and both prior distributions produce similar results. Next we use trajectories generated with steady state density profiles. Figure~\ref{f:PM_S} shows the influence of 
	initial guess and prior mean for three different regimes: an outflux regime with $a = 0.4, \, b = 0.2$ (left panel), 
	an influx limited regime with $a=0.2, \, b = 0.4$ (middle panel) and a maximal current regime with $a = 0.9, \, 
	b = 0.975$ (right panel). We use the same prior distributions and initial guesses as in Figure~\ref{f:PM_TD}.
	\begin{figure}[h!]
		\includegraphics[width=\linewidth]{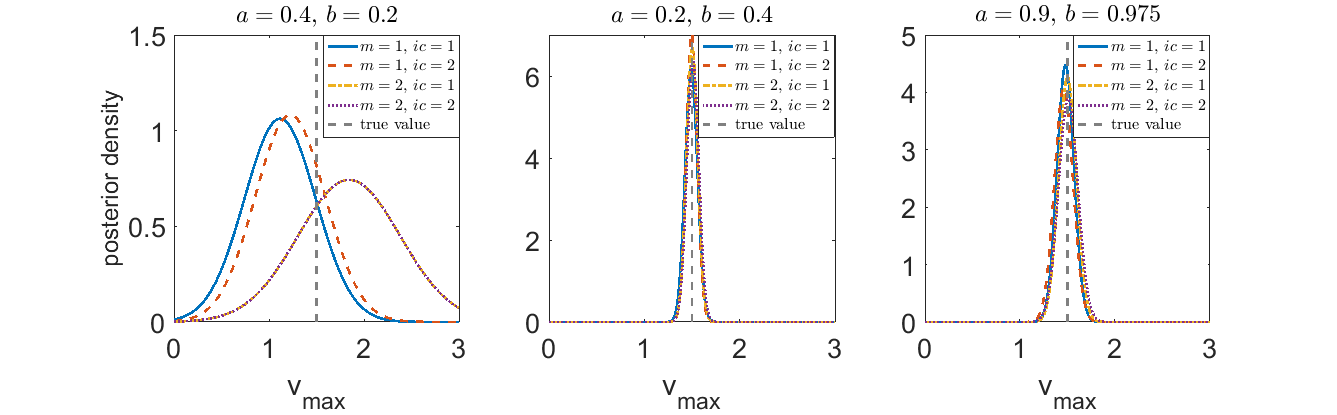}
		\caption{Influence of prior distribution and initial guess for three different regimes using solutions of the steady state equation. 
			Left: outflux limited regime with $a=0.4, b=0.2$. Middle: influx limited regime with $a=0.2, \, b = 0.4$. Right: maximal current regime with $a = 0.9, \, b = 0.975$.}\label{f:PM_S}
	\end{figure}
	We find perfect agreement between the Nelder--Mead and pCN estimators, and the Bayesian estimator has converged 
	to the posterior distribution. The maximal current (right panel) and influx limited (left panel) regimes produce similar 
	results as the corresponding time dependent case; however, in the outflux limited regime we can observe that the 
	posterior distribution stays close to its prior distribution.
	The reason for this may be found in the trajectories used -- in the outflux limited regimes, individuals experience very high densities
	and get `stuck'. Since they can not move, they do not experience the maximum velocity and carry little information about it. 
	In the influx limited regimes, the overall density is smaller and does not influence the velocity that much. This
	can be seen from the depicted trajectories in 
	Figure \ref{f:steady-trajectories}. Note that we do not observe similar problems using 
	time dependent density profiles, since most trajectories experience significant
	velocities before the flow has equilibrated.
	\begin{figure}[h!]
		\includegraphics[width=\linewidth]{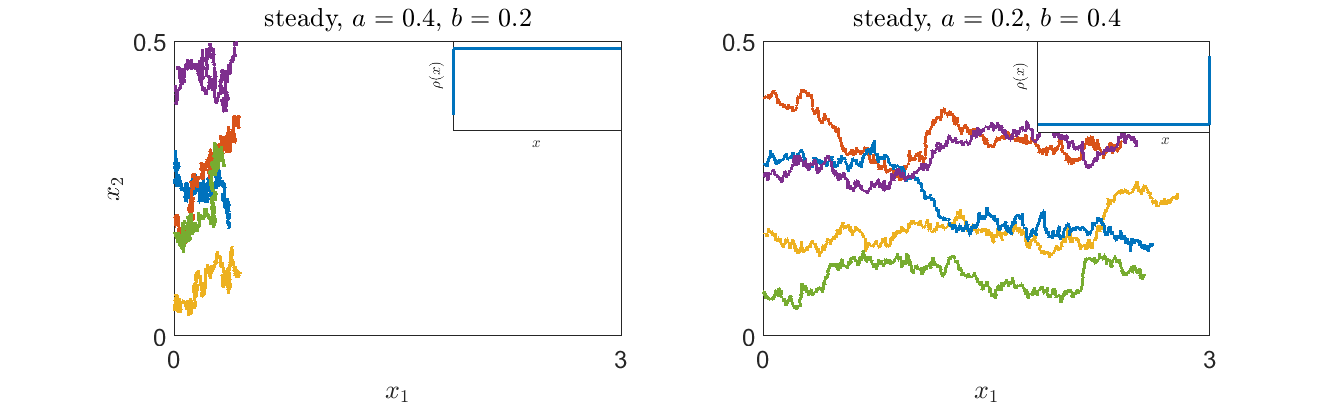}
		\caption{Examples of five trajectories for steady state densities. Left: outflux limited regime. Right: influx limited regime.}\label{f:steady-trajectories}
	\end{figure}

	\paragraph*{Influence of the prior variance} The influence of the prior variance is depicted in Figure~\ref{f:PV_TD} for the time dependent 
	case and in Figure~\ref{f:PV_S} for steady states.  Again, we test this for an outflux (left panels) and influx (right panels) limited regimes. As before, we observe that the estimate for $v_{\max}$ is accurate for both regimes in the time dependent case, except when the prior variance is too small, while for the steady state case we observe again that the outflux limited case produces a posterior distribution which mimics the prior. 
	\begin{figure}[h!]
		\includegraphics[width=\linewidth]{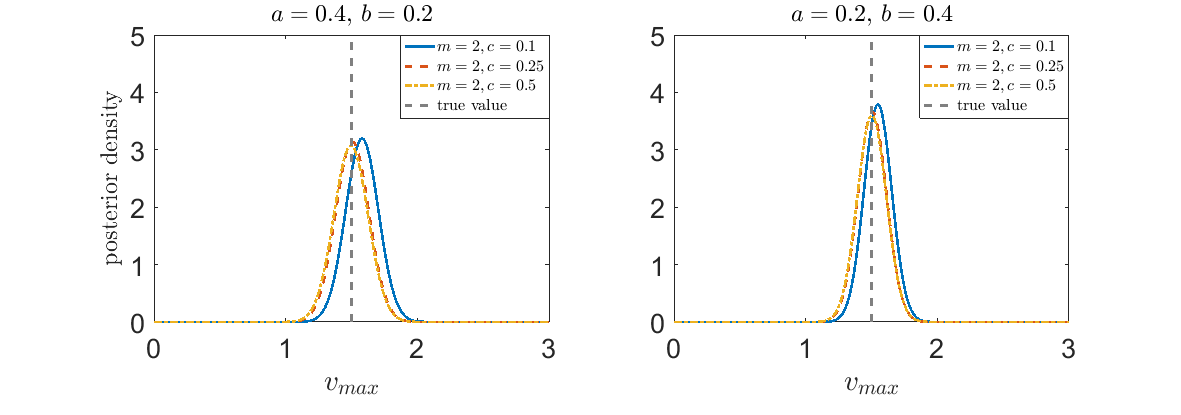}
		\caption{Influence of prior variance for two different regimes using solutions of the time dependent Fokker--Planck equation. Left: outflux limited regime with $a=0.4, b=0.2$. 
			Right: influx limited regime with $a=0.2, \, b = 0.4$.}\label{f:PV_TD}
	\end{figure}
	
	\begin{figure}[h!]
		\includegraphics[width=\linewidth]{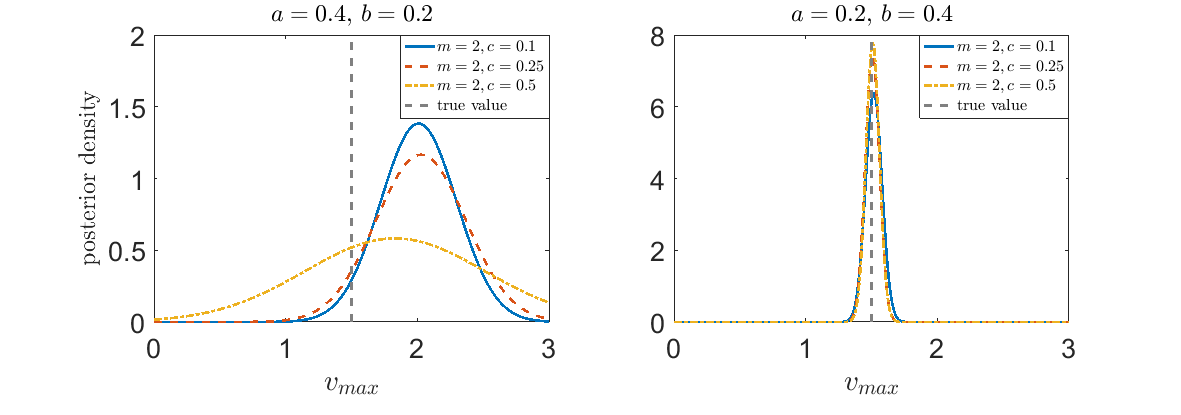}
		\caption{Influence of prior variance for two different regimes using solutions of the steady state equation. Left: outflux limited regime with $a=0.4, b=0.2$. 
			Right: influx limited regime with $a=0.2, \, b = 0.4$.}\label{f:PV_S}
	\end{figure}
	
	\paragraph*{Influence of information and the proximity of the parameters $a$ and $b$} 
	Next we discuss how the amount of information, that is the number of trajectories, and the ratio of the parameters 
	$a$ and $b$ influence the quality of the estimates. The ratio of $a$ and $b$ determines the location 
	of the boundary layer and the density range experienced by the trajectories. For example, the pair $a=0.2, \, b = 0.4$
	gives us a steady state profile with values for $\rho_s$ varying from $[0.2, 0.5]$, while $a = 0.1, \, b = 0.15$ gives 
	values in $[0.075,0.65]$.
	
	We will show that the number of trajectories $J$ plays 
	a large influence when $a\approx b$, while this is not the case when these parameters differ by a ``large" amount.
	In this set of experiments only we use the same value of $\Sigma$ to both
	generate the data and in the inference method employed: $\sigma_1 = \sigma_2 = \sigma_0 = 0.05$.
	Figures~\ref{f:traj_TD} and \ref{f:traj_S} illustrate the role of $J$
	and of the relative size of $a$ and $b$. In each of the panels, we vary the total number of trajectories, that is $J=5, \, 10, \, 15, \, 20$, and study 
	the effect in two influx limited and regimes. 
	\begin{figure}[h!]
		\includegraphics[width=\linewidth]{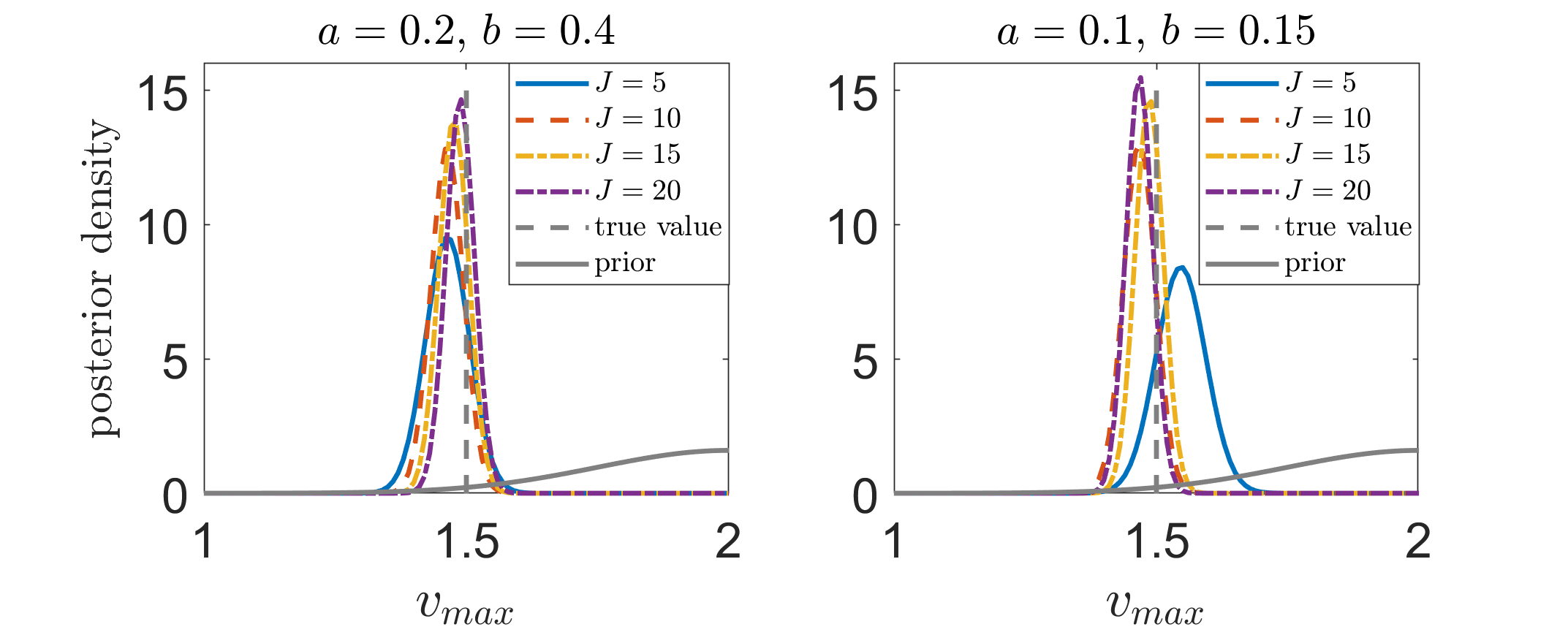}
		\caption{Influence of the number of trajectories on the estimate for $v_{\max}$ for influx limited regimes and in the time dependent case. 
			Left: $a=0.2, b=0.4$. Right: $a=0.1, b=0.15$.}\label{f:traj_TD}
	\end{figure}
	
	We observe that, especially in the regime when the ratio between $a$ and $b$ is 
	further from $1$, the estimates become better as the number of trajectories increases, both in the estimate itself (mean of the posterior distribution) and the variance of the posterior distribution. A similar behavior is observed for the steady state case, which is presented in Figure~\ref{f:traj_S} below.
	\begin{figure}[h!]
		\includegraphics[width=\linewidth]{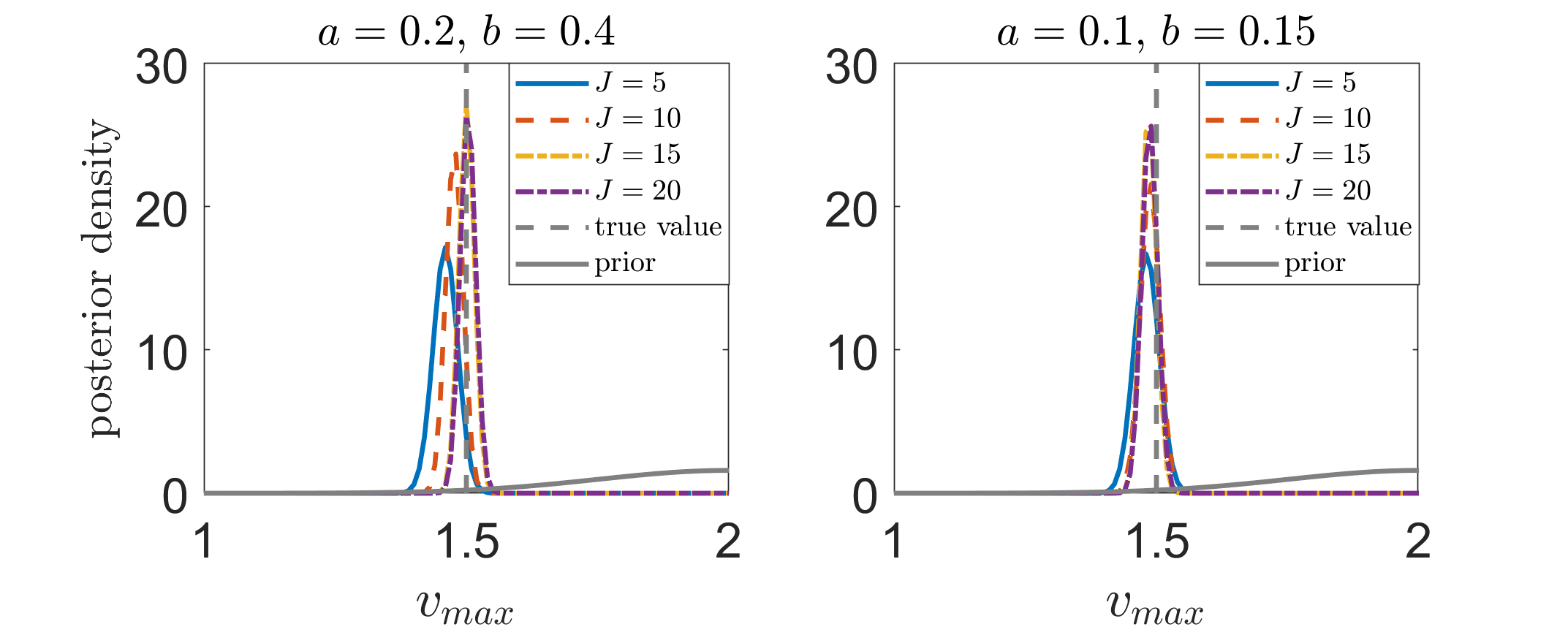}
		\caption{Influence of the number of trajectories on the estimate for $v_{\max}$ for influx limited regimes, in the steady state regime. 
			Left: $a=0.2, b=0.4$. Right: $a=0.1, b=0.15$. }\label{f:traj_S}
	\end{figure}
	
	Again, the estimate becomes closer to the true value, and with a smaller variance, as the number of trajectories increases. Furthermore, the regime where $a/b$ is 
	further from one has better estimates. This suggests that when planning experiments, these are preferable regimes to consider.
	We point out that Figures~\ref{f:traj_TD} and~\ref{f:traj_S} are zoomed in so that the effect of the number of trajectories is clearly seen and, in particular,
	the prior is supported on a much larger length scale than that displayed,
	demonstrating that it has been forgotten.

	\paragraph*{Time dependent vs steady state regimes}
	We have already pointed out that in a steady state situation, outflux limited regimes do not carry information about 
	$v_{\max}$ due to the trajectories getting stuck. However, an important thing to notice is that in the influx limited and 
	maximal current regimes the estimates for $v_{\max}$ are always accurate, and more importantly, less uncertain.  This can be observed in, e.g., Figures~\ref{f:PM_TD} and~\ref{f:PM_S}.

	\subsection{Corridor with a bottleneck}\label{s:bottleneck}
	We conclude with a more realistic example, in which  pedestrians move through a corridor with a bottleneck. The corridor has length $L = 3$ and a maximum width of $2\ell = 0.5$,
	which reduces to $2w = 0.1$ inside the bottleneck. The exit boundary is now split into a rigid wall 
	(Neumann/reflecting boundary conditions) and a door (Robin/partially reflecting boundary conditions) which has width $0.3$. We assume that
	diffusion in the vertical direction is smaller due to limited space inside the bottleneck. In particular we set $\sigma_1 = 0.05$ and $\sigma_2 = 0.03$. 
	In this example the vector $e_1$ in \eqref{eq:F} is replaced by the negative gradient of a potential $\phi$.
	This potential is computed from the eikonal equation with suitable boundary conditions, see \cite{QZZ2007}. The function $\phi$ corresponds to the shortest distance to the exit, and its negative gradient to the optimal trajectory to navigate towards a door and/or through a bottleneck, see Figure \ref{f:potential}. 
	
	\begin{figure}[h!]
		\begin{center}
			\includegraphics[width=0.9\linewidth]{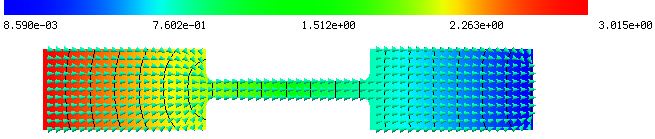}
			\caption{Solution of the eikonal equation in the corridor with a bottleneck described in this section. The arrows depict the direction of the negative gradient, which is the direction with which individuals move.}\label{f:potential}
		\end{center}
	\end{figure}

	We generate trajectories using a time step $\Delta t_{SDE} = 10^{-3}$ and set the final time $T=1$. In the PDE solver the time step was set to $\Delta t_{PDE} =  4\times 10^{-3}$.
	The true value of $v_{\max}$ was, as before $v_{\max} = 1.5$, and we used two of the previous regimes: 
	$a = 0.2, \, b = 0.4$ and $a = 0.4, \, b = 0.2$.
	\begin{figure}[h!]
		\includegraphics[width=\linewidth]{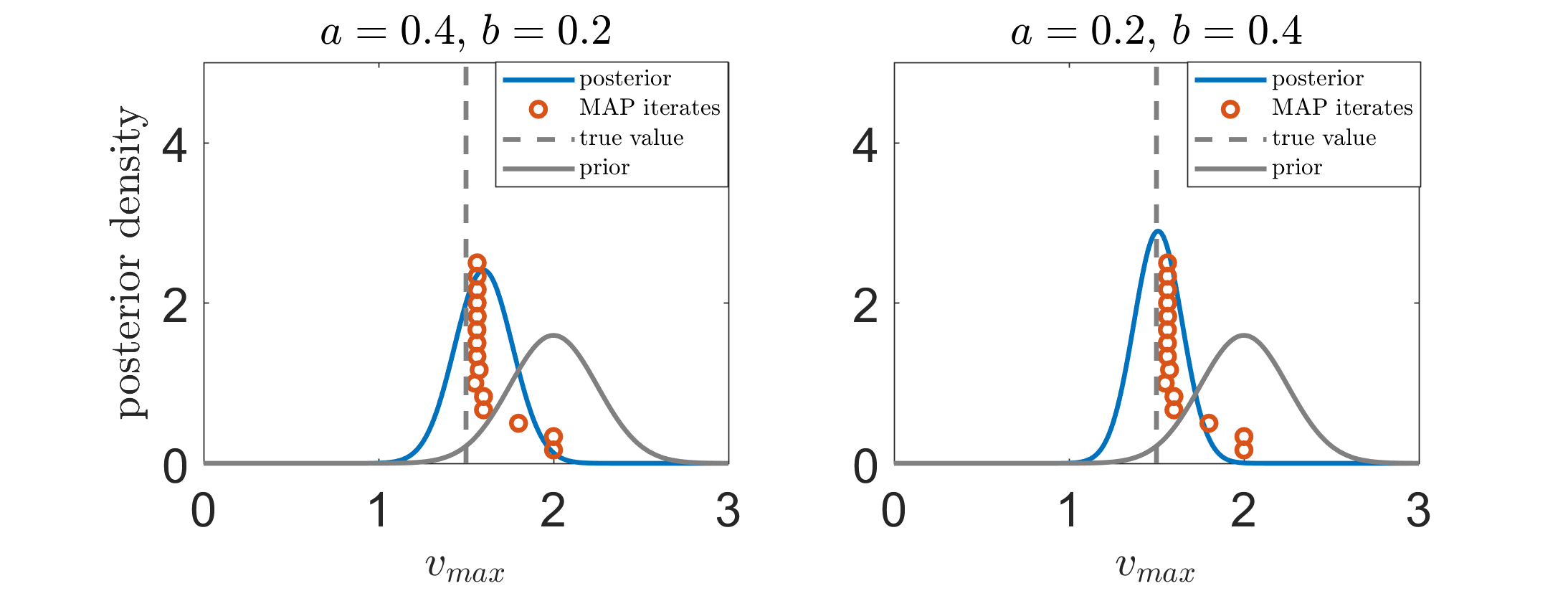}
		\caption{Results of the parameter estimation methodology in a corridor with a bottleneck. Left: outflux limited regime with $a = 0.4, \, b = 0.2$. Right: influx limited regime with $a = 0.2, \, b = 0.4$.}\label{f:bottleneck}
	\end{figure}
	
	In Figure~\ref{f:bottleneck} we present our numerical results. The figure depicts the true value of $v_{\max}$ (dashed 
	gray line), prior distribution (full gray line), posterior distribution (full blue line) and the Nelder--Mead algorithm iterates 
	(orange circles). We observe that both regimes produce very good estimates for $v_{\max}$. This demonstrates that the framework we have developed can be applied to more realistic setups; we intend to study the impact of parameters and geometry in more detail in future work.

	\section{Conclusions/Discussion}
	\label{sec:conclusion}
	We have studied a macroscopic model for a unidirectional flow of pedestrians in a corridor. The evolution 
	of the pedestrian density is given by a nonlinear Fokker--Planck equation, whose coefficients depend on the 
	so-called fundamental diagram. We formulated and analyzed the problem of estimating two parameters of interest in the 
	fundamental diagram using individual trajectories. We assume that these trajectories are realizations of a 
	generalized McKean--Vlasov equation. This identification problem was solved using derivative-free methodologies.
	We have shown that the first parameter -- the maximum pedestrian density -- cannot be estimated from the model considered. 
	The second characteristic quantity -- the maximum speed -- can be accurately learned from a variety of inflow and 
	outflow conditions, both in time dependent and steady state settings. 
	We have also seen that boundary conditions play an important role. We 
	believe that the proposed framework may help  
	to understand their impact in estimation, as well as experimental design.
	
	We also believe that the developed framework provides the basis for future developments in applied mathematics as well as transportation research. In particular, the next steps include the identification of parameters in different forms of the fundamental diagram, or the application of nonparametric estimation techniques to learn its functional form. 
	Furthermore, we want to use pedestrian trajectory data, which requires the framework to be generalized to noisy observations.
	
	These future developments will contribute to the validation of  the fundamental diagram adopted in many models in the transportation literature. It will also justify its use in certain parameter regimes for microscopic pedestrian models.

	\vspace{0.1in}
	
	\noindent{\bf Acknolwedgements} The work of S.G. and A.S. was supported  by the
	EPSRC Programme Grant EQUIP. The work of M.-T.W. and AS was supported
	by a Royal Society international collaboration grant. M.-T. W. acknowledges partial support from the Austrian Academy of Sciences via the New Frontier's grant NST-001. SG is grateful to Imperial College London for
	use of computer facilities.
	The authors are grateful to Grigorios Pavliotis for helpful discussions.
	
	\vspace{0.1in}

	\bibliographystyle{plain}
	\bibliography{ref}

	\section*{Appendix}
	
	\subsection{Proof of Theorem \ref{t:globalexistence} and Corollary \ref{c:regularity}}
	
	Equation \eqref{e:FP-rescaled} is a gradient flow with respect to the entropy functional
	\begin{align}\label{e:entropy}
		E(\rho) = \int_{\Omega} (\rho \log \rho + (1-\rho) \log (1-\rho) - \rho V) dx
	\end{align}
	with $V(x) = x_1$.  The corresponding entropy variable $u = \frac{\delta E}{\delta \rho} = \log \rho - \log (1-\rho) - V$ allows us to write \eqref{e:FP-rescaled} as
	\begin{align*}
		\partial_t \rho = \Div(m(\rho) \nabla u)
	\end{align*}
	with a nonlinear mobility $m(\rho) = \rho (1-\rho)$. This gradient flow structure provides the necessary a-priori estimates to prove global in time existence of solutions.
	
	\begin{proof}[Proof of Theorem \ref{t:globalexistence}]
		\noindent Let $N \in \mathbb{N}$ and consider a discretization of $(0,T]$ into subintervals $(0,T] = \cup_{k=1}^N ((k-1) \tau, k\tau]$ with time steps
		$\tau = \frac{T}{N}$.  We look for a weak solution $\rho: \Omega \times [0,T] \rightarrow \mathcal{S}$ of \eqref{e:FP-rescaled} in the sense of \eqref{eq:weakFP}.
		\noindent The proof is based on the implicit Euler discretization, which gives us a recursive sequence of elliptic problems. We consider its regularized version
		\begin{align}\label{e:timediscreteFPE}
			\frac{\rho_k - \rho_{k-1}}{\tau} = \Div(m(\rho_k) \nabla u_k) + \tau \Delta u_k.
		\end{align}
		Existence of at least one weak solution $\rho \in H^1(\Omega) \cap L^{\infty}(\Omega)$ with $0 \leq \rho \leq 1$ to \eqref{e:timediscreteFPE} follows from Theorem 3.5 in \cite{Burger2016}, if assumptions \ref{a:domain}-\ref{a:potential} are satisfied.
		Note that the transformation from $\rho$ to the entropy variables $u$ is one to one and given by
		\begin{align}
			\rho = \frac{e^{u+V}}{1+e^{u+V}}.
		\end{align}
		Therefore solutions $\rho$ lie automatically in the set $\mathcal{S}$. Furthermore the entropy density
		\begin{align*}
			h(\rho) = \rho \log \rho - (1-\rho) \log(1-\rho) - \rho V
		\end{align*}
		is strictly convex for $\rho \in \mathcal{S}^0$ since $h'(\rho) = \log\frac{\rho}{1-\rho} - V$  and $h''(\rho) = \frac{1}{\rho} + \frac{1}{1-\rho}.$
		Since $h$ is convex  we have that $h(q_1) - h(q_2) \leq h'(q_1) (q_1 - q_2)$ and therefore for $q_1 = \rho_k$ and $q_2 = \rho_{k-1}$ that
		\begin{align}\label{e:convex}
			h(\rho_k) - h(\rho_{k-1}) \leq h'(\rho_k) (\rho_k - \rho_{k-1}).
		\end{align}
		
		\noindent \textit{Entropy dissipation} We consider the weak formulation of \eqref{e:timediscreteFPE} to obtain
		\begin{align*}
			\frac{1}{\tau} \int_{\Omega} (\rho - \rho_{k-1}) \varphi~ dx + \int_{\Omega} \nabla \varphi^ T m(\rho_k) \nabla u_k~ dx -& a \int_{\Gamma_{in}} (1-\rho_k) \varphi ~ ds + \\
			& b \int_{\Gamma_{out}} \rho_k \varphi~ ds + \tau\int_{\Omega} \nabla \varphi \nabla u_k = 0, 
		\end{align*}
		for $\varphi \in H^1(\Omega)$ and use that $\rho_k = h'^{-1}(u_k)$. 
		Since $u_k = h'(\rho_k)$ we can rewrite \eqref{e:convex} as
		\begin{align}\label{e:convexity}
			\frac{1}{\tau} \int_{\Omega} (\rho_k - \rho_{k-1}) u_k dx \geq \frac{1}{\tau} \int_{\Omega} [h(\rho_k) - h(\rho_{k-1})] dx.
		\end{align}
		Using \eqref{e:convexity} and choosing test functions $\varphi = u_k$ gives
		\begin{align*}
			\int_{\Omega} h(\rho_k) dx + \tau\int_{\Omega} &\nabla u_k^T m(\rho_k) \nabla u_k~ dx - a \tau \int_{\Gamma_{in}} (1-\rho_k) u_k ~ ds +\\
			+ &b \tau \int_{\Gamma_{out}} \rho_k u_k~ ds + \tau^2\int_{\Omega} \nabla u_k \nabla u_k \leq \int_{\Omega} h(\rho_{k-1}) dx . 
		\end{align*}
		
		\noindent Next we prove the following a-priori estimate:
		\begin{lemma}
			Let $\rho \in L^2(\Omega)$ and $\rho \in \mathcal{S}^0$ a.e. be such that $u = h'(\rho) \in H^1(\Omega)$. Then there exist constants such that
			\begin{align*}
				\int_{\Omega} \nabla u^T m(\rho) \nabla u - a \int_{\Gamma_{in}} (1-\rho) u ~ds + b \int_{\Gamma_{out}} \rho u~ ds \geq \int_{\Omega} \lvert \nabla \rho \rvert^2 dx  - C
			\end{align*}
		\end{lemma}
		\begin{proof}
			For the first term we obtain
			\begin{align*}
				\int_{\Omega} \nabla u^T m(\rho) \nabla u ~ dx  = \int_{\Omega} \rho (1-\rho) \lvert \nabla u \rvert^2 dx.
			\end{align*}
			Since $\nabla u = \frac{\nabla \rho}{\rho(1-\rho)} - \nabla V$ we have
			\begin{align*}
				\int_{\Omega} \nabla u^T m(\rho) \nabla u dx &=  \int_{\Omega} \frac{\lvert \nabla \rho \rvert^2}{\rho(1-\rho)} dx - 2 \int_{\Omega} \nabla V \cdot \nabla \rho + \int_{\Omega} \rho(1-\rho) \lvert \nabla V\rvert^2 dx = \\
				&\geq  \int_{\Omega} \frac{\lvert \nabla \rho \rvert^2}{2\rho (1-\rho)} - \rho(1-\rho) \lvert \nabla V\rvert^2 dx\\
				& \geq 2 \int_{\Omega} \lvert \nabla \rho \rvert^2 dx - \frac{1}{4} \int_{\Omega} \lvert \nabla V \rvert^2 dx.
			\end{align*}
			The estimates follow from Cauchy's inequality and the fact that $\rho(1-\rho) \leq \frac{1}{4}$. The inflow boundary term can be rewritten as
			\begin{align*}
				- \int_{\Gamma_{in}} (1-\rho) u ds &= -\int_{\Gamma_{in}} (1-\rho) (\log \rho - \log(1-\rho) - V) ds ¸\\
				& = \int_{\Gamma_{in}} \left[(1-\rho) \log \frac{1-\rho}{\rho} + 2 \rho - 1 \right]ds + \int_{\Gamma_{in}} \left[(1-\rho) V - 2 \rho + 1 \right] ds.
			\end{align*}
			The first term on the right hand side is a Kullback--Leibler distance and therefore non-negative. Since $V \in H^1(\Omega)$ we know that its trace satisfies 
			$V \mid_{\partial \Omega} \in L^2(\partial \Omega)$
			and since $\rho \in \mathcal{S}$ the second term is bounded.
			We can use a similar argument for the outflow term, which can be written as
			\begin{align*}
				- \int_{\Gamma_{out}} \rho u ~ds &= -\int_{\Gamma_{out}} \rho (\log \rho - \log (1-\rho) - V)~ ds = \\
				& = \int_{\Gamma_{out}} \left[\rho \log \frac{1-\rho}{\rho} + 2 \rho -1 \right] ~ds + \int_{\Gamma_{out}} \left[\rho V -2 \rho + 1\right] ds.
			\end{align*}
			Again we have a non-negative and a bounded term on the right hand side. 
		\end{proof}
		The dissipation inequality and the recursion yields
		\begin{align}
			\int_{\Omega} h(\rho_k) dx  + \tau \sum_{j=1}^k \int_{\Omega} \lvert \nabla \rho_k \rvert^2 dx  + \tau^2 \sum_{k=1}^n R(u_k, u_k) \leq \int_{\Omega} h(\rho_0) dx + T C.
		\end{align}
		where $R(u,\varphi) = \int_{\Omega} \nabla u \nabla \varphi dx$.
		This discrete entropy dissipation relation allows us to pass to the limit $\tau \rightarrow 0$.\\
		
		\noindent \textit{The limit $\tau \rightarrow 0$:} Let $\rho_k$ denote  a sequence of solutions to \eqref{e:timediscreteFPE}. We define $\rho_{\tau}(x,t) = \rho_k(x)$ for $x \in \Omega$
		and $t \in ((k-1)\tau, k\tau]$
		Then $\rho_{\tau}$ solves the following problem where $\sigma_{\tau}$ denotes the shift operator, that 
		is $(\sigma_{\tau} \rho_{\tau})(x,t) = \rho_{\tau}(x,t-\tau)$ for $\tau \leq t \leq T$,
		and 
		\begin{align}\label{e:weaktimeshift}
			\begin{split}
				\int_0^T \int_{\Omega}& \left(\frac{1}{\tau} (\rho_{\tau} - \sigma_{\tau} \rho_{\tau}) \varphi + \nabla \rho_{\tau} \nabla \Phi - \rho_{\tau}(1-\rho_{\tau}) \nabla \varphi \right)dx \\ 
				&\qquad \qquad - a \int_0^T \int_{\Gamma_{in}} (1-\rho_{\tau}) \varphi ds + b \int_{\Gamma_{out}} \rho_{\tau} \varphi ds= 0.
			\end{split}
		\end{align}
		for test functions $\varphi \in L^2(0,T; H^1(\Omega))$. Then the entropy dissipation inequality becomes
		\begin{align}
			\int_{\Omega} h(\rho_{\tau}(T)) dx  + \int_0^T \int_{\Omega} \lvert \nabla \rho_\tau \rvert^2 dx  + \tau \int_0^T  R(u_k, u_k) \leq \int_{\Omega} h(\rho_0) dx + T C.
		\end{align}
		which gives the a-priori estimate $\lVert \rho_{\tau} \rVert_{L^2(0,T; H^1(\Omega))} \leq K$. Using the a-priori estimates we use \eqref{e:weaktimeshift} to obtain
		\begin{align*}
			\int_0^T \int_{\Omega} \langle \rho_{\tau} - \sigma_{\tau}\rho, \varphi \rangle dt \leq K \lVert \varphi \rVert_{L^2(0,T; H^1(\Omega))}.
		\end{align*}
		We know that $\rho_{\tau}$ is in $L^2(0,T; H^1(\Omega))$ and $\frac{1}{\tau}\langle \rho_{\tau} - \sigma_{\tau} \rho_{\tau}\rangle$  in $L^2(0,T; H^1(\Omega)^*)$. Therefore we can use 
		Aubin's lemma to conclude the existence of a subsequence, also denoted by $\rho_{\tau}$ such that for $\tau \rightarrow 0$
		\begin{align}\label{e:strongconvrho}
			\rho_{\tau} \rightarrow \rho \quad \text{ strongly in } L^2(0,T;L^2(\Omega)).
		\end{align}
		Finally we check that all terms in \eqref{e:weaktimeshift} converge to the right limit as $\tau \rightarrow 0$. Because of \eqref{e:strongconvrho} we know that
		\begin{align*}
			1-\rho_{\tau} \rightarrow 1-\rho \quad \text{ strongly in } L^2(0,T, L^2(\Omega)).
		\end{align*}
		Since $V \in H^1(\Omega)$ we can pass to the limit in the boundary terms as well. This concludes the existence proof.
	\end{proof}

	\begin{corollary}\label{c:regularity}
		Let all assumptions of Theorem \ref{t:globalexistence} be satisfied. Then every weak solution $\rho \in L^2(0,T; H^1(\Omega))$ is also in $C(\bar{\Omega} \times (0,T))$.
	\end{corollary}
	\begin{proof}[Proof of Corollary \eqref{c:regularity}]
		The assertion follows from a bootstrap argument. Since $\rho \in L^2(0,T; H^1(\Omega))$ we can rewrite equation \eqref{e:FP-rescaled} as
		\begin{align*}
			\begin{aligned}
				\partial_t \rho(x,t) - \Div(\Sigma \nabla \rho) &= -\Div(h) & \text{ in } \Omega \times (0,T)\\
				\rho(x,t) &= g(x) & \text{ in } \Gamma \times (0,T)
			\end{aligned}
		\end{align*}
		where $h = \rho F(\rho)  \in L^2(0,T; L^2(\Omega))$ and $g \in L^2(0,T; L^2(\Gamma))$ then 
		\begin{align*}
			\lVert \rho \rVert_{L^{\infty}(0,T; L^2(\Omega))} \leq \lVert h \rVert_{L^2(0,T; L^2(\Omega))} + \lVert g \rVert_{L^2(0,T;L^2(\Gamma))}.
		\end{align*}
		Since $\Delta V = 0$ we can rewrite \eqref{e:FP-rescaled} as
		\begin{align*}
			\begin{aligned}
				\partial_t \rho - \Div(\Sigma \nabla \rho) + v \cdot \nabla \rho &= 0 & \text{ in } \Omega \times (0,T)\\
				(\nabla \rho + v \rho) \cdot n  &= g & \text{ in } \Gamma \times (0,T)\\
				\rho(x,0) &= 0 & \text{ in } \Omega,
			\end{aligned}
		\end{align*}
		with $g \in L^2(0,T; L^q(\Gamma))$. This equation has a unique solution $\rho \in L^q(0,T; W^{2,q}(\Omega)) \cap W^{1,q}(0,T; L^q(\Omega))$. From Sobolev embeddings we deduce that
		$\rho \in C^0(\bar{\Omega} \times (0,T))$.
	\end{proof}
	
\end{document}